\documentclass[a4paper,10pt]{article}

\usepackage{enumerate}
\usepackage{algorithm}
\usepackage{algorithmic}
\usepackage{amsfonts}
\usepackage{amsmath}
\usepackage{amssymb}
\usepackage{amsthm}
\usepackage{bbm}
\usepackage{bm}
\usepackage[mathscr]{eucal}
\usepackage{a4wide}
\usepackage{authblk}
\usepackage{tikz}
\usetikzlibrary{decorations.pathmorphing}
\usepackage[colorlinks=true,allcolors=blue]{hyperref}
\usepackage{doi}
\usepackage[square,numbers]{natbib}

\newcommand{\pr}{\mathbb{P}}								

\newcommand{\Prob}[1]{\pr\left(#1\right)}					
																
\newcommand{\CProb}[2]{\pr\left(\left.#1\right|#2\right)}	

\newcommand{\e}{\mathbb{E}}								

\newcommand{\Exp}[1]{\e\left[#1\right]}					
															
\newcommand{\CExp}[2]{\e\left[\left.#1\right|#2\right]}	

\newcommand{\bigO}[1]{O\left(#1\right)}				
\newcommand{\smallO}[1]{o\left(#1\right)}			
\newcommand{\bigT}[1]{\Theta\left(#1\right)}		

\newcommand{\ind}[1]{\mathbbm{1}_{\left\{#1\right\}}}	

\newcommand{\R}{\ensuremath{\mathbb{R}}}			

\newcommand{\Mcal}{\ensuremath{\mathcal{M}}}		
\newcommand{\Ccal}{\ensuremath{\mathcal{C}}}		
\newcommand{\Ncal}{\ensuremath{\mathcal{N}}}		
\newcommand{\Xcal}{\ensuremath{\mathcal{X}}}		
\newcommand{\Pcal}{\ensuremath{\mathcal{P}}}
\newcommand{\Bcal}{\ensuremath{\mathcal{B}}}

\newcommand{\Gbb}{\ensuremath{\mathbb{G}}}			

\newcommand{\BallM}[2]{\mathcal{B}_\Mcal\left(#1 \, ; \, #2\right)}
\newcommand{\BallN}[2]{\mathcal{B}_N\left(#1 \, ; \, #2\right)}

\newcommand{\Ball}[2]{\mathcal{B}\left(#1 \, ; \, #2\right)}
\newcommand{\BallG}[2]{\mathcal{B}_G\left(#1 \, ; \, #2\right)}
\newcommand{\BallGw}[2]{\mathcal{B}_G^w\left(#1 \, ; \, #2\right)}
\newcommand{\BallGast}[2]{\mathcal{B}_G^\ast\left(#1 \, ; \, #2\right)}
\newcommand{\Ric}{\mathrm{Ric}}						
\newcommand{\vol}{\mathrm{vol}}						
\newcommand{\Po}{\mathrm{Po}}						

\renewcommand{\vec}[1]{\ensuremath{{\bf #1}}}		

\newcommand\numberthis{\addtocounter{equation}{1}\tag{\theequation}}

\newcommand{\dd}{\ensuremath{\mathrm{d}}}

\allowdisplaybreaks

\newtheorem{theorem}{Theorem}[section]
\newtheorem{definition}{Definition}[section]
\newtheorem{lemma}[theorem]{Lemma}
\newtheorem{proposition}[theorem]{Proposition}
\newtheorem{corollary}[theorem]{Corollary}

\newtheorem{remark}[theorem]{Remark}

\title{Ollivier curvature of random geometric graphs \\ converges to Ricci curvature of their Riemannian manifolds}
\author[1,2,3]{Pim van der Hoorn}
\author[4]{Gabor Lippner}
\author[6]{Carlo Trugenberger}
\author[2,3,4,5]{Dmitri Krioukov}

\affil[1]{Department of Mathematics and Computer Science, Eindhoven~University~of~Technology}
\affil[2]{Department of Physics, Northeastern University}
\affil[3]{Network Science Institute, Northeastern University}
\affil[4]{Department of Mathematics, Northeastern University}
\affil[5]{Department of Electrical and Computer Engineering, Northeastern University}
\affil[6]{SwissScientific Technologies}

\begin{document}

\maketitle

\begin{abstract}
Curvature is a fundamental geometric characteristic of smooth spaces. In recent years different notions of curvature have been developed for combinatorial discrete objects such as graphs. However, the connections between such discrete notions of curvature and their smooth counterparts remain lurking and moot. In particular, it is not rigorously known if any notion of graph curvature converges to any traditional notion of curvature of smooth space. Here we prove that in proper settings the Ollivier-Ricci curvature of random geometric graphs in Riemannian manifolds converges to the Ricci curvature of the manifold. This is the first rigorous result linking curvature of random graphs to curvature of smooth spaces. Our results hold for different notions of graph distances, including the rescaled shortest path distance, and for different graph densities. With the scaling of the average degree, as a function of the graph size, ranging from nearly logarithmic to nearly linear. \par \bigskip

 \par \bigskip

\end{abstract}

\section{Introduction}

Curvature is a fundamental concept in the study of geometric spaces. It is a local parameter whose behavior often controls global phenomena on the manifold. In particular, bounds on the Ricci curvature are known to imply an array of properties, including diameter bounds, control of the spectrum, and sub-Gaussian decay of the heat kernel. If the curvature of some space is upper-bounded by a negative value, then such space has a boundary at infinity and some other universal characteristics of (coarsely) hyperbolic spaces.
Unfortunately, most notions of curvature are applicable only to smooth continuous spaces, such as Riemannian and pseudo-Riemannian manifolds. While there exist some \emph{combinatorial} notions of curvature~\cite{cheeger1984curvature,forman2003bochner}, none has the same power as their smooth counterparts. We refer to~\cite{najman2017modern} for a general overview of discrete curvatures. The focus of this paper is graph curvature.

In~\cite{ollivier2007ricci,ollivier2009ricci,ollivier2010survey}, Yann Ollivier introduced a definition of curvature for general metric spaces as a discretization of the well-known Ricci curvature. Since this definition is applicable to any metric space, it is applicable to graphs in particular. Even though relatively recent, it has already proven to be quite influential and fruitful. In analysis of networks, Ollivier-Ricci curvature has been used, for example, to identify communities~\cite{sia2019ollivier}, analyze cancer cells~\cite{sandhu2015graph}, asses the fragility of financial networks~\cite{sandhu2016ricci} and robustness of brain networks~\cite{farooq2019network}, and to embed networks for machine learning applications~\cite{gu2018learning}. Ollivier-Ricci curvature has also been analyzed for several types of (random) graphs including Erd\H{o}s-R\'{e}nyi random graphs~\cite{lin2011ricci}. Some general bounds for this curvature have also been established based on different graph properties~\cite{lin2011ricci,jost2014ollivier,bhattacharya2015exact}. 
These and other applications of Ollivier-Ricci curvature have also stimulated general interest in graph curvature, leading to the introduction and studies of many other notions of graph curvature~\cite{sreejith2016forman,liu2018bakry,kempton2019large,cushing2019long}. 

An interesting aspect of Ollivier-Ricci curvature (or any other notion of discrete curvature) is that it creates a bridge between geometry and discrete structures. For example, discrete curvatures play an important role in the field of manifold learning where the discrete objects are data points lying on some manifold, and the task is to learn from the data the properties of the manifold~\cite{ache2019ricci}. 

A related task is that of graph embedding: given a graph, find its embedding in a smooth space such that graph distances between nodes are approximated by distances in the space. Curvature has proven to be important for finding the right space to embed the graph into~\cite{gu2018learning}.

In addition to these classical applications, geometry has also proven to be an important and powerful concept for designing latent-space models of random graphs whose properties---such as degree distributions, clustering, distance distributions---closely resemble those of real-world networks~\cite{krioukov2010hyperbolic,jacob2015spatial,krioukov2016clustering,bringmann2019geometric}. These relations between geometry and network properties inevitably lead to the question whether characteristics of latent geometries of networks can be inferred from discrete properties of graphs that represent these networks. Since curvature is a fundamental characteristic of geometry, it is a natural first candidate for uncovering latent geometry in networks. Hence, a proper notion of graph curvature is needed, a notion that would be known to converge to the true curvature of the geometric space underlying the graph, if it exists. 

Quantum gravity is yet another area where convergence of graph curvature is of interest. Here one wants to find a discrete geometry that converges in the continuum limit to the geometry of physical spacetime. To this end, Ollivier-Ricci curvature and its variations have been extensively investigated recently~\cite{klitgaard2018introducing,trugenberger2017combinatorial,cunningham2020dimensionally}.

Despite the interest in Ollivier-Ricci and related curvatures of discrete and combinatorial spaces, the fundamental question of convergence remains largely open. That is, does there exist a discrete notion of curvature that converges in some limit to a traditional notion of curvature of smooth spaces. In general, such convergence may be too much to wish for since, for instance, it is known that there cannot exist any discrete version of Gaussian curvature that would converge on any triangulation of any smooth space~\cite{xu2005convergence}.  

There are, however, some positive results in this direction. One is for the convergence of an angle-defect-based notion of curvature of smooth triangulations of Riemannian manifolds~\cite{cheeger1984curvature}. Another one is a manifold learning method designed for consistent estimation of Ricci curvature of a submanifold in Euclidean space based on a point cloud sprinkled uniformly onto the submanifold~\cite{ache2019ricci}. Perhaps the closest result to ours is the one in~\cite{benincasa2010scalar,belenchia2016continuum} where a discrete version of the d'Alembertian operator is defined for causal sets in 2- and 4-dimensional Lorentzian manifolds. This discrete d'Alembertian is then shown to converge to the traditional d'Alembertian in the continuum limit. To the best our knowledge, there currently exist no general convergence results for truly combinatorial objects in general and random graphs in Riemannian manifolds in particular.

In this paper we study the question of convergence of Ollivier-Ricci curvature of graphs. We consider random geometric graphs whose nodes are a Poisson process in a Riemannian manifold and whose edges are formed only between nodes that lie within a given distance threshold from each other in the manifold. We show that as the size of such graphs tends to infinity, their Ollivier-Ricci curvature recovers the Ricci curvature of the underlying manifold. To the best of our knowledge, this is the first result that relates a discrete notion of curvature of graphs to the continuum version of curvature of their underlying geometry.

The remainder of the paper is structured as follows. In the next Section~\ref{sec:notations_definitions} we introduce the basic notations and definitions needed to present our main results.  We present these results in Section~\ref{sec:main_results}. That section ends with some general comments and outlook. We then provide a general overview of the proof strategy in the first half of Section~\ref{sec:proof_overview}. The second half of that section contains the proofs of the main results. The final Section~\ref{sec:proofs} contains all the remaining details and proofs of intermediate results that are skipped in Section~\ref{sec:proof_overview}.

\section{Notations and definitions}\label{sec:notations_definitions}

\subsection{Geometric graphs} 

Given a metric space $(\mathcal{X}, d)$, a countable \emph{node set} $X \subseteq \mathcal{X}$, and \emph{connection radius} $\varepsilon > 0$, we define $G(X,\varepsilon)$ as the graph whose nodes are all the elements in $X$. An edge between $x, y \in X$ exists if and only if $d(x,y) \le \varepsilon$. Since the nodes of $G$ are points in the metric space, we will refer to them using $x$ and $y$, instead of indices $i$ and $j$, and write $x \in G$ if $x$ is a node of $G$.

We will also use $G_{xy}$ to denote the indicator of an edge between $x$ and $y$ in $G$ and define $\Ncal_x$ to be the neighborhood of node $x$, i.e.
\[
	\Ncal_x = \left\{y \in G \, | \, G_{xy} = 1\right\}.
\]
Note that $\Ncal_x = X \cap \Ball{x}{\varepsilon}$, where $\Ball{x}{\varepsilon}$ denotes the closed ball around $x \in \mathcal{X}$ of radius $\varepsilon$ with respect to the distance $d$, but excluding $x$.

\subsection{Random geometric graphs}

In this paper we consider graphs that are constructed by randomly placing points in the metric space $\Xcal$, according to a Poisson process. In order to analyze a notion of curvature on these graphs we need to impose some additional structure on $\Xcal$. More precisely, we will consider Riemannian manifolds as the spaces on which graphs are constructed. We briefly recall some notions of Riemannian geometry needed for the setup and refer the reader to~\cite{jost2009geometry,oneill1983semiriemannian} for more details on the topic. 

Formally, a Riemannian manifold is a pair $(\Mcal, g)$ where $\Mcal$ is a smooth manifold and for each $x \in \Mcal$, $g_x$ is a smooth inner product on the tangent space $T_x\Mcal$ at $x$. This inner product induces a metric $d_\Mcal$, called the Riemannian metric. Since we are mainly interested in metric spaces, we denote a Riemannian manifold by the pair $(\Mcal, d_\Mcal)$. 

Throughout the remainder of this paper we work with Riemannian manifolds that are \emph{orientable}, \emph{connected} and \emph{complete}. The first property ensures that there exists a globally defined \emph{volume form} $\vol_\Mcal$ on $\Mcal$ so that we can perform integration on $\Mcal$. For any $U \subseteq \Mcal$ we will write $\vol_\Mcal(U) = \int_U \dd \vol_\Mcal$ to denote the volume of $U$. The second property says that, as a topological space, $\Mcal$ cannot be separated into the union of disjoint open sets. Finally, completeness means that for any two points $x, y \in \Mcal$ there exists a shortest path (geodesic) in $\Mcal$ connecting $x$ and $y$, whose length is $d_\Mcal(x,y)$. We also note that if $\Mcal$ is connected and compact, then it is complete. With this setup we can define a random geometric graph on a Riemannian manifold in an analogous way to classic random geometric graph in Euclidean space.

\begin{definition}\label{def:random_riemannian_graph}
Let $(\mathcal{M}, d_\Mcal)$ be a smooth, orientable, connected and compact $N$-dimensional Riemannian manifold. Fix $\varepsilon > 0$ and consider a Poisson process $\Pcal_n$ on $\Mcal$ with intensity measure $\frac{n}{\vol_\Mcal(\Mcal)} \, \dd \vol_\Mcal$. Then we define the \emph{random geometric graph} $\Gbb_{n}(\varepsilon) := G(\Pcal_n, \varepsilon)$.
\end{definition}

\begin{remark}[Conditions on the manifold]
From a technical perspective, we only need the manifold to be smooth. This is because we will be working on shrinking neighborhoods of some fixed point $x^\ast \in \Mcal$. For a sufficiently small neighborhood $U$, we can always construct a volume form that is well defined on $U$ and ensure that every two points in $U$ are connected by a geodesic path. We could then fix a sufficiently small and compact neighborhood $\mathcal{C}$ of $x^\ast$ and then consider a Poisson process on $\mathcal{C}$ with intensity measure $n/\vol_\Mcal(\mathcal{C}) \, \dd \vol_\Mcal$. 

The only difference with the global setup is that we would need to frame everything in terms of sufficiently small neighborhoods and deal with possible boundary issues in our proofs. In the end, since curvature is a local property, these issues would vanish. Still, framing all results in this \emph{local} setting would add additional technical layers to the proofs. For convenience, we therefore choose to present everything in terms of global and nice requirements on the manifold.
\end{remark}

We shall next introduce a notion of curvature on random geometric graphs. Since curvature is inherently a local property, it makes sense to define curvature on graphs as a property of an edge. For our analysis we will take a more general approach and consider curvature between two fixed nodes in the graph that are connected by a path. We then analyze its behavior as the size of the graph tends to infinity. 

For any $x \in \Mcal$, we write $\Gbb_n(x,\varepsilon) := G(X_n, \varepsilon)$, where $X_n = \{x\} \cup \Pcal_n$. That is, $\Gbb_n(x,\varepsilon)$ is a random geometric graph with $x$ added to the node set. Similarly, for any pair of points $(x,y) \in \Mcal$ we write $\Gbb_n(x, y,\varepsilon) := G(X_n^\prime, \varepsilon)$, with $X_n^\prime = \{x, y\} \cup \Pcal_n$.  We refer to both $\Gbb_n(x,\varepsilon)$ and $\Gbb_n(x, y,\varepsilon)$ as \emph{rooted random graphs}.

\subsection{Ollivier-Ricci curvature on graphs}

The definition of Ollivier-Ricci curvature uses the Wasserstein metric (transportation distance), which we shall introduce next. Recall that a coupling between two probability measures $\mu_1$ and $\mu_2$ is a joint probability measure $\mu$
whose marginals are $\mu_1$ and $\mu_2$.

\begin{definition}
Let $\mu_1$ and $\mu_2$ be probability measures on a metric space $(\Xcal, d)$ and let $\Gamma(\mu_1, \mu_2)$ denote
the set of all couplings $\mu$ between $\mu_1$ and $\mu_2$. Then the Wasserstein metric (Kantorovich-Rubenstein distance of order one) is given by
\begin{equation}\label{eq:def_wasserstein_inf}
	W_1(\mu_1, \mu_2) = \inf_{\mu \in \Gamma(\mu_1, \mu_2)} \int_{\Xcal \times \Xcal} d(x,y) \, d\mu(x,y)
\end{equation}
\end{definition}

Let $G$ be a graph. The definition of Ollivier-Ricci curvature on graphs relies on two ingredients, a metric on $G$ and a family of probability measures, indexed by the vertices. 

\begin{definition}\label{def:Ollivier_triple}
An Ollivier-triple $\mathcal{G}$ is a triple $(G, d_G, {\bf m})$, where $G$ is a graph, $d_G$ a metric on $G$ and ${\bf m} = \{m_x\}_{x \in G}$ a family of probability measures on $G$ for each node $x \in G$.
\end{definition}

Given an Ollivier-triple $\mathcal{G} = (G, d_G, {\bf m})$, we write $W_1^{\mathcal{G}}$ for the Wasserstein metric with respect to the metric space $(G, d_G)$. We then define for any pair of nodes $x, y \in G$ the associated Ollivier curvature as
\begin{equation}\label{eq:def_olivier_curvature}
	\kappa(x,y;\, \mathcal{G}) = \begin{cases}
		1 - \frac{W_1^{\mathcal{G}}(m_x, m_y)}{d_G(x,y)} &\mbox{if } d_G(x,y) < \infty \\
		0 &\mbox{otherwise.}
	\end{cases}
\end{equation}

\begin{remark}
\hfil\\
\begin{enumerate}
\item The concept of Ollivier-Ricci curvature is not restricted to graphs and can be defined on any metric space where we have a sequence of probability measures. A specific example of these are Riemannian manifolds $(\Mcal, d_\Mcal)$.
\item Note that a sequence $\{m_x\}_{x\in G}$ of probability measures on $G$ gives rise to a random walk on the graph. The transition probabilities are given by $\CProb{x_{t+1} \in A}{x_t = x} = m_x(A)$. So an Ollivier-triple consists of a graph, a metric and a random walk on the graph. However, since we will only use concepts related to the measures $m_x$ we refrain from using any random walk terminology.
\item When $d_G$ is the shortest path metric on $G$ and ${\bf m}$ corresponds to the uniform probability measures on the neighborhoods $\Ncal_x$, i.e. $m_x(y) = G_{xy}/|\Ncal_x|$, we are in the classic setting for Ollivier-Ricci curvature on graphs~\cite{jost2014ollivier,ni2015ricci,paeng2012volume}. In this work, however, we shall use different combinations of metrics on graphs and probability measures to obtain our results. This is why we define Ollivier-Ricci curvature on graphs in a more general way.
\item The reason why we set $\kappa(x,y;\, \mathcal{G}) = 0$ if the nodes are not in the same connected component is because we work with random graphs and this way we ensure that $\kappa(x,y;\, \mathcal{G})$ is a real-valued random variable.
\end{enumerate}
\end{remark}

\subsection{Curvature in Riemannian manifolds}

Our main results relate the standard Ricci curvature of a manifold to the Ollivier-Ricci curvature of the random geometric graph constructed on this manifold.  For this, we briefly recall the definition of the Ricci curvature, see~\cite{jost2009geometry,oneill1983semiriemannian}.

In general, the curvature of a geometric space is intended as a local measure for how ``different" a region of the space is from that of the flat Euclidean space. Notions of curvature in Riemannian geometry are governed by the \emph{Riemannian curvature tensor} $R$. Given an $N$-dimensional Riemannian Manifold $(\mathcal{M}, d_\Mcal)$, a point $x \in \mathcal{M}$ and two vectors $\vec{v}, \vec{w} \in T_x \mathcal{M}$ (the tangent space of $x$), the Riemannian curvature tensor with respect to $\vec{v}$, $\vec{w}$ is a linear map $R(\vec{v},\vec{w}) : T_x \Mcal \to T_x \Mcal$, written as $\vec{u} \mapsto R(\vec{v}, \vec{w}) \vec{u}$ and defined in terms of the Levi-Civita connection on the tangent bundle. It quantifies to what extent the manifold $\Mcal$ is not diffeomorphic with flat Euclidean space. 

In this paper we will use the notion of curvature called \emph{Ricci curvature}. For two vectors $\vec{v}$ and $\vec{w}$, the Ricci curvature $\mathrm{Ric}(\vec{v}, \vec{w})$ is defined, in terms of the Riemannian tensor, as the trace of the linear map
\[
	\vec{u} \mapsto R(\vec{u},\vec{v})\vec{w}, \quad \vec{u} \in T_x \mathcal{M}.
\]
Given a point $x \in \mathcal{M}$ and a unit vector $\vec{v} \in T_x \mathcal{M}$, we often refer to $\mathrm{Ric}(\vec{v},\vec{v})$ as the \emph{Ricci curvature} of $x$ with respect to $v$. 

This Ricci curvature is related to another notion of curvature, called \emph{sectional curvature}, which is defined as
\[
	K(\vec{v},\vec{w}) = \frac{\langle R(\vec{v},\vec{w})\vec{v}, \vec{w} \rangle}
	{\langle \vec{v}, \vec{v} \rangle \langle \vec{w}, \vec{w} \rangle - \langle \vec{v}, \vec{w} \rangle^2},
\]
where $\langle \cdot, \cdot\rangle$ denotes the inner product on the tangent space. One can show that $\Ric(\vec{v},\vec{v})$ is obtained by averaging the sectional curvature $K(\vec{v},\vec{w})$ over all unit vectors $\vec{w} \in T_x\Mcal$.

In the remainder of this paper we will work with the Ricci curvature of a point $x$, with respect to some tangent vector $\vec{v}$. We note that it is not needed to understand the fine details behind curvature of Riemannian manifolds to understand all the details of the results or proofs.

\section{Main results}\label{sec:main_results}

Here we state our results regarding the convergence of Ollivier-Ricci curvature of random geometric graphs on Riemannian manifolds. We note that if the manifold dimension is $N = 1$, then there is nothing to prove, so that we always assume that $N \ge 2$.

We mainly consider two different distances on the graphs, leading to two different but related results. Although we consider different distances on graphs, we shall always consider uniform measures on balls of a certain radius. We shall clearly distinguish between the connection radius of the graph and the radius used for the uniform measures:

\begin{center}
\begin{tabular}{ll}
  Connection radius:    & $\varepsilon_n$ \\
  Measure radius:       & $\delta_n$. 
\end{tabular}
\end{center}

The former is the connectivity distance threshold: if the distance between a pair of nodes in the manifold is below this threshold, then these nodes are connected by an edge in the graph. The latter radius is the radius of the ball (either in the graph or in the manifold) over which the uniform probability measure is distributed. 

Let $G_n = \Gbb_n(\varepsilon_n)$ be a random geometric graph on $\Mcal$ and $d_G$ a distance on $G_n$. Then, for a node $x \in G_n$, we define the graph ball of radius $\lambda$ around $x$ as
\[
	\BallG{x}{\lambda} := \left\{y \in G_n \setminus\{x\} \, : \, d_G(x,y) \le \lambda\right\}.
\]
Note that $\BallG{x}{\lambda}$ depends on the definition of the graph distance $d_G$. For our results we consider Ollivier-triples $\mathcal{G}_n = (G_n, d_G, {\bf m}^G)$, were ${\bf m}^G$ are the uniform measures on $\BallG{x}{\delta_n}$, i.e.
\begin{equation}
	m_x^G(y) = \begin{cases}
		\frac{1}{|\BallG{x}{\delta_n}|} &\mbox{if } y \in \BallG{x}{\delta_n}\\
		0 &\mbox{else.}	
	\end{cases}
\end{equation}

We reiterate that if $\varepsilon_n = \delta_n$ and the graph metric $d_G$ is the shortest path distance, then we are in the classical setting of Ollivier-Ricci curvature on graphs as considered in the past literature~\cite{jost2014ollivier,ni2015ricci,paeng2012volume}. 

\subsection{Graphs with manifold weighted distance}

Let $G_n = \Gbb_n(x^\ast, \varepsilon_n)$ be a random rooted graph on $\Mcal$. Then we define the manifold weighted graph distance $d_G^w$ as the weighted shortest-path distance on $G_n$ where each edge $(u,v)$ is assigned weight $d_\Mcal(u,v)$, corresponding to the distance between the nodes on the manifold. Similar to $\BallG{x}{\lambda}$, we denote by $\BallGw{x}{\lambda}$ the graph ball of radius $\lambda$ with respect to $d_{G}^w$ and let ${\bf m}^{G, w} = (m_x^{G,w})_{x \in G}$ denote the uniformly measures on the balls $\BallGw{x}{\delta_n}$. 

Finally, given a point $x \in \Mcal$ and a vector $\vec{v} \in T_x \Mcal$, we say that another point $y \in \Mcal$ is \emph{at distance $\delta$ in the direction of $\vec{v}$}, if $d_\Mcal(x,y) = \delta$ and $y$ lies on the geodesic starting at $x$ in the direction of $\vec{v}$.

Our first result shows that for certain combinations of connection radius $\varepsilon_n$ and measure radius $\delta_n$, the Ollivier-Ricci curvature on $G_n$ converges to the Ricci curvature. 



\begin{theorem}\label{thm:convergence_ollivier_almost_sparse_graphs_weighted}
Let $N \ge 2$, $(\Mcal, d_\Mcal)$ be a smooth, orientable, connected and compact $N$-dimensional Riemannian manifold, $x^\ast \in \mathcal{M}$ and $\vec{v}$ a unit tangent vector at $x^\ast$. Furthermore, let $\varepsilon_n = \bigT{\log(n)^a n^{-\alpha}}$, $\delta_n = \bigT{\log(n)^b n^{-\beta}}$ (as $n \to \infty$) where the constants satisfy
\begin{align*}
	0 < \beta \le \alpha, \quad
	\alpha + 2\beta \le \frac{1}{N},
\end{align*}
and $a \le b$ if $\alpha = \beta$ and $\min\{a, a + 2b\} > \frac{2}{N}$ if $\alpha + 2\beta = \frac{1}{N}$. 

Let $y_n^\ast \in \Mcal$ be at distance $\delta_n$ in the direction of $\vec{v}$ and $G_n = \Gbb_n(x^\ast, y_n^\ast, \varepsilon_n)$ be rooted random graphs on $\Mcal$. Then for the Ollivier-triple $\mathcal{G}_n^w = (G_n,d_{G_n}^w, {\bm m}^{G,w})$, it holds
\[
	\lim_{n \to \infty}  \Exp{\left|\frac{2(N + 2)\kappa(x^\ast, y_n^\ast;\, \mathcal{G}_n^w)}{\delta_n^2} -\Ric(\vec{v}, \vec{v})\right|} = 0.
\]
\end{theorem}

Theorem~\ref{thm:convergence_ollivier_almost_sparse_graphs_weighted} relates two different quantities. The first is the Ollivier-Ricci curvature in the graph between the node $x^\ast$ and another node $y_n^\ast$ that is at distance $\delta_n$ from $x^\ast$ in the direction of vector $\vec{v}$. The second is the Ricci curvature of the manifold at $x^\ast$ in the $\vec{v}$-direction. The theorem says that if we properly rescale the former, it converges in expectation to the latter.

\begin{remark}
\hfill\\
\begin{enumerate}
\item Note that Theorem~\ref{thm:convergence_ollivier_almost_sparse_graphs_weighted} states that $\delta_n^{-2} \, 2(N + 2)\kappa(x^\ast, y_n^\ast;\, \mathcal{G}_n^w)$ converges in the $L^1$ sense to $\Ric(\vec{v}, \vec{v})$. In particular, this implies the concentration result
\[
	\lim_{n \to \infty} \Prob{\left|\frac{2(N + 2)\kappa(x^\ast, y_n^\ast;\, \mathcal{G}_n^w)}{\delta_n^2} -\Ric(\vec{v}, \vec{v})\right| \ge \epsilon} = 0, \quad \text{for all } \epsilon > 0.
\]
\item Since $\varepsilon_n, \delta_n \to 0$, both the connectivity and measure neighborhoods of $x^\ast$ become smaller as $n$ grows. Indeed, curvature is a local property, so that measuring it more accurately requires smaller regions. 
\item While the connectivity neighborhood of $x^\ast$ is shrinking, the expected number of $x^\ast$'s neighbors lying in it is growing with $n$. To see this, note that for large enough $n$ the volume of the ball $\BallM{x}{\varepsilon_n}$ around $x \in \Mcal$ can be approximated by that of the $N$-dimensional Euclidean ball. Hence, for any $x \in \Gbb_n(x^\ast, y_n^\ast, \varepsilon_n)$, as $n \to \infty$
\[
	\Exp{|\Ncal_x|} = n \mathrm{vol}_\Mcal\left(\BallM{x}{\varepsilon_n}\right) = \Theta\left(n \varepsilon_n^N\right)
	= \Theta\left((\log(n))^{aN} n^{1-\alpha N}\right).
\]
The conditions of the theorem imply that $\alpha \le \alpha + 2\beta \le \frac{1}{N}$, so that $1 - \alpha N \ge 0$. This means that the average degree diverges faster than logarithmically if $\alpha N < 1$. More generally, the conditions of Theorem~\ref{thm:convergence_ollivier_almost_sparse_graphs_weighted} imply that the average degree always diverges faster than $\log(n)^2$.
\end{enumerate}
\end{remark}

If we consider the classic setting where the connection and measure radii are the same, $\varepsilon_n = \delta_n$, then the following result is a direct consequence of Theorem~\ref{thm:convergence_ollivier_almost_sparse_graphs_weighted}.

\begin{corollary}\label{cor:convergence_ollivier_dense_graphs_weighted}
Let $N \ge 2$, $(\Mcal, d_\Mcal)$ be a smooth, orientable, connected and compact $N$-dimensional Riemannian manifold, $x^\ast \in \mathcal{M}$ and $\vec{v}$ a unit tangent vector at $x^\ast$. Furthermore, let $\delta_n = \bigT{\log(n)^b n^{-\beta}}$, with $\beta \le \frac{1}{3N}$ and $b > \frac{2}{N}$ whenever $\beta = \frac{1}{3N}$. 

Let $y_n^\ast \in \Mcal$ be at distance $\delta_n$ in the direction of $\vec{v}$ and $G_n = \Gbb_n(x^\ast, y_n^\ast, \delta_n)$ be rooted random graphs on $\Mcal$. Then for the Ollivier-triple $\mathcal{G}_n^w = (G_n,d_{G_n}^w, {\bm m}^{G,w})$, it holds
\[
	\lim_{n \to \infty}  \Exp{\left|\frac{2(N + 2)\kappa(x^\ast, y_n^\ast;\, \mathcal{G}_n^w)}{\delta_n^2} -\Ric(\vec{v}, \vec{v})\right|} = 0.
\]
\end{corollary}

While the conditions in this corollary imply that the average degree in $\Gbb_n(x^\ast, y_n^\ast, \delta_n)$ diverges faster than $n^{2/3}$, Theorem~\ref{thm:convergence_ollivier_almost_sparse_graphs_weighted} works for graphs where the average degree can be almost as small as $\log(n)^2$. The crucial component for establishing the curvature convergence in graphs with so much smaller average degree is to consider different connection and measure radii and let the connection radius decrease at a faster rate than the measure radius, i.e.\ $\varepsilon_n \ll \delta_n$.

\begin{remark}[Extreme cases for convergence of curvature]
Corollary~\ref{cor:convergence_ollivier_dense_graphs_weighted} covers one set of extreme casse for the combination $a, b, \alpha$ and $\beta$ from Theorem~\ref{thm:convergence_ollivier_almost_sparse_graphs_weighted}, were we take $\beta$ to be as big as possible. This means that we compute the curvature using uniform probability measures on a set of nodes that is \emph{as small as possible}. For the true extreme case, let $\epsilon > 0$ be arbitrarily small and define $\beta = \frac{1 - \epsilon}{3N}$ and $b = \frac{2 +\epsilon}{N}$. Then, to calculate the curvature, we need to compute the Wasserstein metric between uniform probability measures on neighborhoods that contain
\[
	\bigT{n \varepsilon^N} = \bigT{n \delta_n^N} = \bigT{\log(n)^{2 + \epsilon} n^{\frac{2 + \epsilon}{3}}},
\]
number of nodes. The consequence, however, is that our graphs have average degree diverging at the same rate: $\log(n)^{2 + \epsilon} n^{\frac{2 + \epsilon}{3}}$.

In order to get graphs whose average degree diverges as slow as possible, we need to consider an other extreme case. Again let $\epsilon > 0$ be arbitrary small. Now we define
\[
	a = \frac{2 + \epsilon}{N}, \quad \alpha = \frac{1 - \epsilon}{N}, \quad b = a \quad \text{and } \beta = \frac{\epsilon}{2N}.
\]
For these choices we have that $\alpha + 2\beta = 1/N$ and $\min\{a, a + 2b\} = a > 2/N$ so that the result from Theorem~\ref{thm:convergence_ollivier_almost_sparse_graphs_weighted} holds. In this case, the average degree scales as
\[
	\bigT{n \varepsilon_n^N} = \bigT{ \log(n)^{Na} n ^{1 - N\alpha}} = \bigT{\log(n)^{2 + \epsilon} n^{\epsilon}},
\]
which is almost logarithmic. However, we now need to compute the Wasserstein metric with respect to the uniform measure on a number of nodes that scales as
\[
	\bigT{n \delta_n^N} = \bigT{\log(n)^{2 + \epsilon} n ^{1 - \epsilon/2}}.
\]
That is, in order to compute curvature on graphs with almost logarithmic average degree, we need to consider the uniform probability measure on almost the entire graph.
\end{remark}

\subsection{Graphs with hop count distance}

In the previous section we considered Ollivier-Ricci curvature of graphs on Riemannian manifolds, with graph edges weighted by manifold distances. These weights encode a lot of information about the manifold metric structure, so that one may feel not terribly surprised that we can recover manifold curvature from graph curvature using this information. The natural question is then if it is possible to prove convergence of Ollivier-Ricci curvature based on shortest path distances $d_G^s$ in \emph{unweighted} graphs. It turns out that this can be done under some slightly more restrictive conditions on the connection and measure radii. 

For this we define, for any random geometric graph $G_n = \Gbb(\varepsilon_n)$, the rescaled shortest path distance $d_G^\ast(x,y) = \varepsilon_n d_G^s(x,y)$. Similar to the previous setting we let $\BallGast{x}{\delta_n}$ denote the balls of radius $\delta_n$ around in $x \in G_n$ with respect to the metric $d_G^\ast$ and define the random walk measures
\[
	m^{G,\ast}_x(y) = \begin{cases}
		\frac{1}{|\BallGast{x}{\delta_n}|} &\mbox{if } y \in \BallGast{x}{\delta_n}\\
		0 &\mbox{else}.
	\end{cases}
\]

\begin{theorem}\label{thm:convergence_ollivier_graphs_hopcount}
Let $(\Mcal, d_\Mcal)$ be a smooth, orientable, connected and compact $2$-dimensional Riemannian manifold, $x^\ast \in \mathcal{M}$ and $\vec{v}$ a unit tangent vector at $x^\ast$. Furthermore, let $\varepsilon_n = \bigT{\log(n)^a n^{-\alpha}}$, $\delta_n = \bigT{\log(n)^b n^{-\beta}}$ where the constants satisfy
\[
	0 < \beta \le 1/9 \quad \text{and} \quad 3\beta \le \alpha \le \frac{1-3\beta}{2},
\]
and $a < 3b$ if $\alpha = 3\beta$ and $2a + 3b > 1$ if $\alpha = \frac{1-3\beta}{2}$. 

Let $y_n^\ast \in \Mcal$ be at distance $\delta_n$ in the direction of $\vec{v}$ and $G_n = \Gbb_n(x^\ast, y_n^\ast, \varepsilon_n)$ be rooted random graphs on $\Mcal$. Then for the Ollivier-triple $\mathcal{G}_n^\ast = (G_n,d_{G_n}^\ast, {\bm m}^{G,\ast})$, it holds
\[
	\lim_{n \to \infty}  \Exp{\left|\frac{2(N + 2)\kappa(x^\ast, y_n^\ast;\, \mathcal{G}_n^\ast)}{\delta_n^2} -\Ric(\vec{v}, \vec{v})\right|} = 0.
\]
\end{theorem}

\begin{remark}
\hfill\\
\begin{enumerate}
\item Note that unlike Theorem~\ref{thm:convergence_ollivier_almost_sparse_graphs_weighted}, here we do not include any information on the distances between nodes on the manifold. We only need the connection radius.
\item Observe that the theorem allows to select an $\alpha$ that is arbitrary close to $\frac{1}{2}$. In particular,
\[
	\Exp{|\Ncal_x|} = \bigT{n \varepsilon_n^2} = \bigT{\log(n)^{2a} n^{1 - 2\alpha}} 
	\le \bigT{\log(n)^{2a} n^{6\beta}}.
\]
Hence by selecting a small $\beta$ we have a discrete notion of curvature that converges on graphs with almost logarithmic average degree, without using any information on the manifold.
\item Theorem~\ref{thm:convergence_ollivier_graphs_hopcount} currently only works in $2$-dimensional manifolds. This is because the proof relies on results for the stretch (the fraction $d_G/d_\Mcal$) for random geometric graphs in $2$-dimensional Euclidean space~\cite{diaz2016relation}. Our proof techniques, however, immediately allow the results to be extended to higher dimensions, once similar types of stretch results for these spaces are obtained.
\end{enumerate}
\end{remark}

\subsection{Summary, comments, caveats, and outlook}

In summary, we have proven that upon proper rescaling, the Ollivier-Ricci curvature of random geometric graphs on a Riemannian manifold converges to the Ricci curvature of the underlying manifold.

Our first result, Theorem~\ref{thm:convergence_ollivier_almost_sparse_graphs_weighted}, establishes convergence of Ollivier-Ricci curvature for a wide range of connectivity and measure radii. In particular, it contains as a corollary the classical setting where both radii are the same, Corollary~\ref{cor:convergence_ollivier_dense_graphs_weighted}. The theorem does, however, require knowledge of pairwise distances between connected nodes in the manifold. 

Our second result, Theorem~\ref{thm:convergence_ollivier_graphs_hopcount}, relaxes this requirement and establishes the same convergence without any knowledge of distances in the manifold. This does come at the price of slightly more restrictive conditions on the possible connection and measure radii. Still, as for the first result, the convergence holds al the way up to graphs whose average degree grows very slowly (almost logarithmically).

To the best of our knowledge, these are the first rigorous results on the convergence of a discrete notion of curvature of random combinatorial objects to a traditional continuum notion of curvature of smooth space. 

While the classical setting for Ollivier-Ricci graph curvature uses probability measures (random walks) on balls of the same radius as the graph connection radius, in this paper we allow the radii to be different. This is an important generalization. In particular, we find that in order for the curvature to converge on graphs with almost logarithmic average degree, we need the probability measure radius to be much larger than the connection radius. This is intuitively expected because in order to ``feel'' any curvature in graphs with such a low density, we really need to consider large ``mesoscopic'' neighborhoods in them since otherwise all we could see is local ``microscopic'' Euclidean flatness. It would be interesting to see how this more general approach would generalize known results for the classical setting of Ollivier-Ricci curvature of graph families that have been investigated in the past, such as trees or Erd\H{o}s-R\'{e}nyi random graphs~\cite{bhattacharya2015exact,jost2014ollivier}. 

In our recent numeric experiments~\cite{hoorn2020ollivier}, we have seen that in manifold-distance-weighted random geometric graphs, the Ollivier-Ricci curvature convergence holds even for graphs with constant average degree. Unfortunately, the proof techniques presented in this paper do not allow for a direct generalization to this setting. Therefore, other techniques are needed to (dis)confirm the convergence of Ollivier-Ricci curvature of graphs with constant average degree. We note that one definitely cannot expect Ollivier-Ricci curvature to converge in all possible graph sparsity settings. For example, we definitely need the giant component to exist to talk about any curvature convergence.

For the task of learning latent geometry in networks, our results can still be improved, particularly by removing the requirement to know the connection radius. When presented just with a truly unweighted realization of a random geometric graph, this radius needs first to be learnt, estimated. It would thus be interesting to see if convergence would still hold if we replace the true value of the connection radius with its consistent estimation, e.g.\ based on the average degree. Here we expect the speed of curvature convergence (if any) to depend on the speed of estimator convergence in a possibly nontrivial way.

Finally, now that we have seen that Ollivier-Ricci curvature of random combinatorial discretizations of smooth spaces converges to their Ricci curvature, it would be interesting to investigate whether such convergence also holds for other popular notions of discrete curvature. Forman-Ricci curvature~\cite{sreejith2016forman} appears to be a good next candidate for such investigation.

\section{Proof overview}\label{sec:proof_overview}

Our main results in Theorems~\ref{thm:convergence_ollivier_almost_sparse_graphs_weighted} and~\ref{thm:convergence_ollivier_graphs_hopcount} follow from our more general result on the Ollivier-Ricci curvature convergence in graphs whose edges are always weighted by some weights. That is, we assume that all edges in our graphs always have some weights, assigned according to some scheme. For our general result it is not important what these weights or their assignment scheme are. What is important is that the graph distance $d_G$ between a node pairs is a good approximation of the manifold distance $d_\Mcal$ between the corresponding pair of points. To quantify how good this approximation is, we introduce the following definition.

\begin{definition}\label{def:delta_good_approx}
Let $(\Mcal, d_\Mcal)$ be a $N$-dimensional Riemannian manifold and $G_n = \Gbb_n(x^\ast, \varepsilon_n)$ a rooted random graph on $\Mcal$. A graph distance $d_G$ on $G_n$ is said to be a \emph{$\delta_n$-good approximation of $d_\Mcal$} if $d_\Mcal \le d_G$ and the following holds (as $n \to \infty$): there exits a $Q > 3$ and $\xi_n = \smallO{\delta_n}$ such that with probability $1 - \smallO{\delta_n^3}$,
\begin{equation}\label{eq:def_distance_approximation}
	\left|d_\Mcal(u,v) - d_G(u,v)\right| \le d_\Mcal(u,v)\xi_n^2 + \xi_n^3,
\end{equation}
holds for all $u,v \in \BallM{x^\ast}{Q \delta_n} \cap G_n$.
\end{definition}

\begin{remark}[Asymptotic expressions]
Most of our results will deal with asymptotic relations, e.g. $\xi_n = \smallO{\delta_n}$. Unless stated otherwise, these asymptotic relations will always be understood as $n \to \infty$. 
\end{remark}

Recall that $\BallG{x}{\delta}$ denotes the set of nodes in the graph that are at graph distance at most $\delta$ from $x$,
\[
	m_x^G(y) = \begin{cases}
		\frac{1}{|\BallG{x}{\delta_n}|}	&\mbox{if } y \in \BallG{x}{\delta_n}, \\
		0	&\mbox{else,}
	\end{cases}
\]
and define
\begin{equation}\label{eq:def_lambda_n}
	\lambda_n = \log(n)^{\frac{2}{N}} n^{-\frac{1}{N}}.
\end{equation}
This $\lambda_n$ will play the role of an additional radius, for extending the graph distance $d_G$ to the manifold. In short, to define a distance between $u, v \in \Mcal$, we will connect $u$ and $v$ to all points of the graph withing radius $\lambda_n$ and then use the graph distance. The radius $\lambda_n$ has been selected such that the expected number of nodes inside any ball $\BallM{x}{\lambda_n}$ is of the order $\bigT{\log(n)^2}$. Hence, the probability of observing no node of the graph inside any such ball is $\bigO{e^{-\log(n)^2}} = \smallO{n^{-1}}$, which is sufficiently small. More details on the use of $\lambda_n$ can be found in Section~\ref{ssec:extending_graph_distance} 

Our general result is then as follows. 

\begin{theorem}\label{thm:convergence_ollivier_graphs_general}
Let $N \ge 2$, $(\Mcal, d_\Mcal)$ be a smooth, orientable, connected and compact $N$-dimensional Riemannian manifold, $x^\ast \in \mathcal{M}$ and $\vec{v}$ a unit tangent vector at $x^\ast$. Furthermore, let $\varepsilon_n \le \delta_n = \smallO{1}$ be such that $\lambda_n = \smallO{\varepsilon_n}$ and $\lambda_n = \smallO{\delta_n^3}$. 

Let $y_n^\ast \in \Mcal$ be at distance $\delta_n$ in the direction of $\vec{v}$, $G_n = \Gbb_n(x^\ast, y_n^\ast, \varepsilon_n)$ be rooted random graphs on $\Mcal$ and $d_G$ a $\delta_n$-good approximation of $d_\Mcal$. Then, if we consider the Ollivier-triple $\mathcal{G}_n = (G_n,d_{G}, {\bm m}^G)$,
\[
	\lim_{n \to \infty}  \Exp{\left|\frac{2(N + 2)\kappa(x^\ast, y_n^\ast;\, \mathcal{G}_n)}{d_G(x^\ast, y_n^\ast)^2} -\Ric(\vec{v}, \vec{v})\right|} = 0.
\]
\end{theorem}

Once we have established this general result, our main results in Theorems~\ref{thm:convergence_ollivier_almost_sparse_graphs_weighted} and~\ref{thm:convergence_ollivier_graphs_hopcount}  follow if we can show that the considered graph distances are $\delta_n$-good approximations.

A key ingredient in the proof of Theorem~\ref{thm:convergence_ollivier_graphs_general} is the convergence result for Ollivier-Ricci curvature for uniform measures on Riemannian manifolds, proved in the seminal paper on the topic~\cite{ollivier2009ricci}. In a high-level overview, our proof approximates Ollivier-Ricci curvature of probability measures on the graph with those on the manifold. Having obtained such an approximation with a required accuracy, we then apply the convergence result from~\cite{ollivier2009ricci}.

Since Ollivier-Ricci curvature is defined by the Wasserstein metric on probability measures, our analysis focuses on approximating the Wasserstein metric of discrete probability measures on the graph by the Wasserstein metric of uniform probability measures on the manifold. This is done in three steps: 1) extend the graph distance $d_G$ to a distance $\widetilde{d}_\Mcal$ on the manifold such that the Wasserstein metric $\widetilde{W}_1$ with respect to this new distance is a good approximation of the Wasserstein metric $W_1$ on the manifold, 2) show that the Wasserstein metric between the probability measure $m_x^G$ on the graph and the discrete probability measure $m_x^\Mcal$ on the nodes within the ball $\BallM{x}{\delta_n}$ is sufficiently small and 3) show that the Wasserstein metric between the uniform measure on $\BallM{x}{\delta_n}$ and the discrete probability measure $m_x^\Mcal$ is sufficiently small. 

\begin{remark}
In all cases, sufficiently small means that the error terms are of smaller order than $\delta_n^3$. This is because the Wasserstein metric is first divided by $\delta_n$ to obtain the curvature, which is then divided by $\delta_n^2$ to make it converge to the Ricci curvature.
\end{remark}

We proceed with explaining all ingredients and the three steps in more detail. We reiterate that unless stated otherwise, we will assume that $\varepsilon_n \le \delta_n$ are two sequences converging to zero such that $\lambda_n = \smallO{\varepsilon_n}$ and $\lambda_n = \smallO{\delta_n^3}$.

\subsection{Ollivier curvature on Riemannian manifolds}

Let $(\Mcal, d_\Mcal)$ be a smooth, orientable, connected and compact $N$-dimensional Riemannian manifold. For $x \in \Mcal$ and $\delta > 0$, we write $\BallM{x}{\delta} \subseteq \Mcal$ to denote the closed ball of radius $\delta$ around $x$, i.e. $\BallM{x}{\delta} = \{y \in \Mcal \, :\ , d_\Mcal(x,y) \le \delta\}$. Recall that
\[
	\vol_\Mcal(\BallM{x}{\delta}) := \int_{\BallM{x}{\delta}} \dd \vol_\Mcal(y),
\]
denotes the volume of the ball $\BallM{x}{\delta}$.

Now fix $\delta > 0$ and consider the uniform measure on balls of radius $\delta$. That is, for $x \in \Mcal$ we take the probability measure $\mu_x^\delta$ given by
\begin{equation}\label{eq:def_uniform_rw_manifold}
	\dd \mu^\delta_x(y) = \begin{cases}
		\frac{1}{\vol_\Mcal\left(\BallM{x}{\delta}\right)} \, \dd \vol_\Mcal(y) &\mbox{if } y \in \BallM{x}{\delta}\\
		0 &\mbox{else.}
	\end{cases}
\end{equation}
We will refer to $\mu^\delta_x$ as the \emph{uniform $\delta$-measure}.

The following result from~\cite{ollivier2009ricci} shows that for a uniform $\delta$-measure on a Riemannian manifold, the Ollivier curvature (properly rescaled) converges to the Ricci curvature as $\delta \to 0$.

\begin{theorem}[Example 7 in \cite{ollivier2009ricci}]\label{thm:convergence_curvature_manifolds}
Let $(\Mcal, d_\Mcal)$ be a smooth complete $N$-dimensional Riemannian manifold $x \in \mathcal{M}$ and $\vec{v}$ a 
unit tangent vector at $x$. Let $\delta > 0$ and $y$ be the point at distance $\delta$ in the direction of $\vec{v}$. Then if we consider the Ollivier-Ricci curvature $\kappa$ for the uniform $\delta$-measures given by \eqref{eq:def_uniform_rw_manifold}: 
\[
	\lim_{\delta \to 0} \frac{2(N+2)}{\delta^2} \kappa(x,y) = \Ric(\vec{v},\vec{v}).
\]
\end{theorem}

\begin{remark}
The result in Theorem~\ref{thm:convergence_curvature_manifolds} clearly exhibits the local nature of curvature as it holds in the limit where the distance $d_\Mcal(x,y) = \delta$ between the two points goes to zero.
\end{remark}

Taking $\delta = \delta_n$, $x = x^\ast$ and $y = y_n^\ast$ in the above theorem, we have that the rescaled Ollivier-Ricci curvature associated to the uniform $\delta_n$-measures converges to the Ricci curvature as $n \to \infty$. The main strategy for proving Theorem~\ref{thm:convergence_ollivier_graphs_general} is to compare this ``uniform" version of the curvature $\kappa$ on the manifold to the discrete version on the graph. 

More precisely, we need to prove that
\begin{equation}\label{eq:proof_overview_main}
	\Exp{\left|W_1^G(m_{x^\ast}^G,m_{y_n^\ast}^G) - W_1(\mu_{x^\ast}^{\delta_n}, \mu_{y_n^\ast}^{\delta_n})\right|}
	= \smallO{\delta_n^3}
\end{equation}
There are two complicating factors here. First, we have to deal with two Wasserstein metrics defined on two different spaces. Second, we have to compare discrete probability measures with continuous ones. We deal with the different Wasserstein metrics in the next section and with comparing the different measures in Section~\ref{ssec:approx_graph_measures} and Section~\ref{ssec:coupling_continuous_discrete_measures}.

\subsection{Extending the graph distance to the manifold}

In order to compare the two different Wasserstein metrics in~\eqref{eq:proof_overview_main} we extend the graph distance $d_G$ to a distance $\widetilde{d}_\Mcal$ defined on a sufficiently large part of $\Mcal$. In particular, we will consider the ball $\BallM{x^\ast}{Q \delta_n}$, with $Q > 3$ from Definition~\ref{def:delta_good_approx}. The extension is such that for any two nodes $x,y \in G_n$, $d_G(x,y) = \widetilde{d}_\Mcal(x,y)$, so that $W_1^G(m_{x^\ast}^G,m_{y_n^\ast}^G)$ can be replaced by the Wasserstein metric associated with $\widetilde{d}_\Mcal$.

Recall the definition of $\lambda_n$ from~\eqref{eq:def_lambda_n}. Take $G_n = \Gbb_n(x^\ast, y_n^\ast, \delta_n)$ and let $U \subset \mathcal{M}$ be a countable set of points. Then we define the graph $G_{n}(U)$ obtained from $G_n$ by adding the points of $U$ to the vertex set and connecting each $u \in U$ to any other node $x \in G_n \setminus U$ for which $d_{\mathcal{M}}(x,u) \le \lambda_n/2$. After this, we assign to each new edge $(u,x)$ the weight $d_{\Mcal}(x,u)(1 + \xi_n^2) + \xi_n^3$, with $\xi_n$ from Definition~\ref{def:delta_good_approx}. We can then extend the graph distance to the manifold by defining $\widetilde{d}_{\mathcal{M}}(u,v)$ to be the graph distance $d_G(u,v)$ computed in the extended graph $G_{n}(\{u,v\})$ with the added weights. Observe that if $x,y \in G_n$ then $\widetilde{d}_{\mathcal{M}}(x,y) = d_G(x,y)$ so that the distance on nodes of $G_n$ does not change and hence $\widetilde{d}_{\mathcal{M}}$ is a true extension of $d_{G}$. In addition, by definition of the graph distance it immediately follows that $\widetilde{d}_{\mathcal{M}}(u,v) = 0$ if and only if $u = v$. Figure~\ref{fig:extended_distance} shows an illustration of the extended distance.

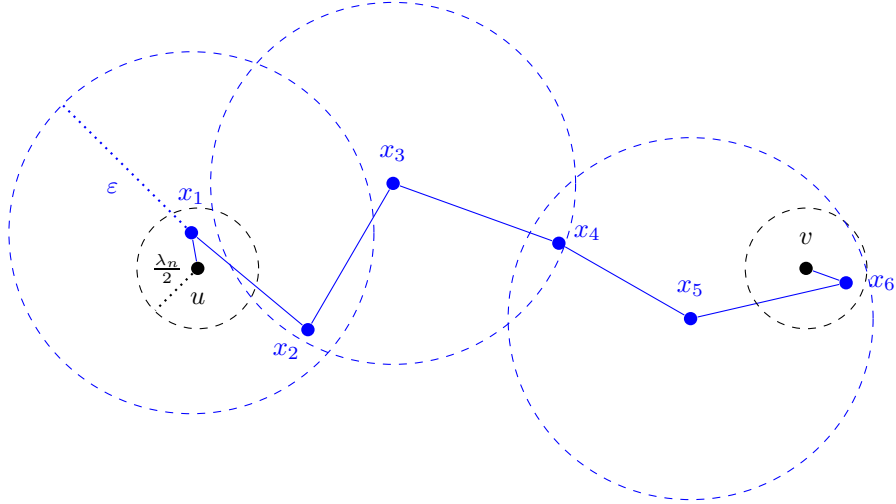
\begin{figure}[!t]
\centering
\begin{tikzpicture}[scale=0.8]
	\tikzstyle{vertex}=[fill, circle, inner sep=0pt, minimum size=5pt]
	\tikzstyle{edge}=[color=black]
	
	\draw node[vertex] (u) at (0,0) {};
	\draw node[vertex] (v) at (10,0) {};
	
	\path (u)+(100:0.6) node[vertex,blue] (x0) {};
	\path (x0)+(320:2.5) node[vertex,blue] (x1) {}; 
	\path (x1)+(60:2.8) node[vertex,blue] (x2) {}; 
	\path (x2)+(340:2.9) node[vertex,blue] (x3) {};
	\path (x3)+(-30:2.5) node[vertex,blue] (x4) {};
	\path (v)+(340:0.7) node[vertex,blue] (x5) {}; 
	
	\path (u)+(270:0.5) node {$u$};
	\path (v)+(90:0.5) node {$v$};

	\path (x0)+(90:0.6) node {\color{blue} $x_1$};
	\path (x1)+(-135:0.5) node {\color{blue} $x_2$};
	\path (x2)+(90:0.5) node {\color{blue} $x_3$};
	\path (x3)+(20:0.5) node {\color{blue} $x_4$};
	\path (x4)+(90:0.5) node {\color{blue} $x_5$};
	\path (x5)+(0:0.6) node {\color{blue} $x_6$};
	
	\draw[blue] (u) -- (x0) -- (x1) -- (x2) -- (x3) -- (x4) -- (x5) -- (v);
	
	\draw[dashed] (u) circle  (1cm);
	\draw[dashed] (v) circle  (1cm);
	
	\draw[dashed,blue] (x0) circle (3cm);
	\draw[dashed,blue] (x2) circle (3cm);
	\draw[dashed,blue] (x4) circle (3cm);

	\path (u)+(225:1) node[inner sep=0pt, minimum size=0pt] (lambda) {};
	\path (u)+(180:0.5) node {\small $\frac{\lambda_n}{2}$};
	\draw[dotted,thick] (u) -- (lambda);
	
	\path (x0)+(135:3) node[inner sep=0pt, minimum size=0pt] (epsilon) {};
	\path (x0)+(150:1.5) node {\color{blue} \small $\varepsilon$};
	\draw[dotted,thick,blue] (x0) -- (epsilon);
	
\end{tikzpicture}
\caption{Illustration of the extended graph distance $\widetilde{d}_\Mcal$. Here $u$ is connected to node $x_1$ and $v$ to $x_6$ and the shortest geodesic-weighted path between $x_1$ and $x_6$ goes over $5$ edges.}
\label{fig:extended_distance}
\end{figure}

It is important to note that this extended distance depends on the random graph $G_n$. Therefore, it could happen that two added points $u,v \in U$ are not connected in $G_n(U)$, i.e. there does not exist a path from $u$ to $v$ in the extended graph. This happens if there are no nodes in $\BallM{u}{\lambda_n/2}$ or in $\BallM{v}{\lambda_n/2}$ or if none of the node pairs $(x,y) \in \BallM{u}{\lambda_n/2} \times \BallM{v}{\lambda_n/2}$ are connected by a path in $G_n$. Therefore, to justify the definition of the extended manifold distance we need to make sure that, with sufficiently high probability, theses situations do not occur. 

\begin{lemma}\label{lem:good_event}
Let $G_n = \mathbb{G}_n(x^\ast, y_n^\ast, \delta_n)$ and $Q > 3$ be the constant from Definition~\ref{def:delta_good_approx}. Then, there exists an event $\Omega_n$ satisfying $\Prob{\Omega_n} \ge 1- \smallO{\delta_n^3}$ such that on this event the following holds:
\begin{enumerate}[\upshape $\Omega$1)]
\item $(\BallM{x^\ast}{Q \delta_n}, \widetilde{d}_\Mcal)$ is a metric space,
\item $\widetilde{d}_\Mcal(u,v) = d_\Mcal(u,v) + \smallO{\delta_n^3}$ and
\item any two nodes $u,v \in \BallM{x^\ast}{Q \delta_n}$ are connected by a path in the graph.
\end{enumerate}
\end{lemma}

The first property ensures that our extended distance is an actual distance. Moreover, by the second property, this extended distance is a good approximation of the true distance on the manifold. Finally, the last property makes sure that $d_G(x^\ast, y_n^\ast) = \widetilde{d}_\Mcal(x^\ast, y_n^\ast) < \infty$, so that the curvature $\kappa$ between $x^\ast$ and $y_n^\ast$ is well-defined and not forced to be zero. The precise definition of $\Omega_n$ is not needed to understand the high level arguments as well as the proof of the main results. For now, let us refer to $\Omega_n$ as the \emph{good event}. Details on this event can be found in Section~\ref{ssec:extending_graph_distance}.

Let $\widetilde{W}_1$ denote the Wasserstein metric with respect to $\widetilde{d}_\Mcal$, which is only well-defined on the \emph{good event} $\Omega_n$. Since the distance is determined by the graph $G_n = \Gbb_n(x^\ast, y_n^\ast, \delta_n)$, the Wasserstein metric is also a random object. The following proposition states that, on the event $\Omega_n$, the difference between the Wasserstein metrics $\widetilde{W}_1$ and $W_1$ is small. The proof is given in Section~\ref{ssec:extending_graph_distance}.

\begin{proposition}\label{prop:L1_error_adjusted_wasserstein_manifold}
Let $G_n = \mathbb{G}_n(x^\ast, \varepsilon_n)$ and $\mu_1, \mu_2$ be two probability measures on $\Mcal$ with support contained in $\BallM{x^\ast}{Q \delta_n}$. Then
\[
	\CExp{\left|\widetilde{W}_1(\mu_1, \mu_2) - W_1(\mu_1, \mu_2)\right|}{\Omega_n} = \smallO{\delta_n^3}.
\]
\end{proposition}

Recall that $\widetilde{d}_\Mcal(x,y) = d_G(x,y)$ if $x,y \in G_n$, and so $W_1^G(m_{x^\ast}^G,m_{y_n^\ast}^G) = \widetilde{W}_1(m_{x^\ast}^G,m_{y_n^\ast}^G)$. Hence, since the uniform $\delta_n$-measures  $\mu_{x^\ast}^{\delta_n}$ and $\mu_{y_n^\ast}^{\delta_n}$ are probability measures on $\Mcal$ with support contained in $\BallM{x^\ast}{Q \delta_n}$, Proposition~\ref{prop:L1_error_adjusted_wasserstein_manifold} implies that on the good event,
\[
	\left|W_1^G(m_{x^\ast}^G,m_{y_n^\ast}^G) - W_1(\mu_{x^\ast}^{\delta_n},\mu_{y_n^\ast}^{\delta_n})\right| = \left|\widetilde{W}_1(m_{x^\ast}^G,m_{y_n^\ast}^G) - \widetilde{W}_1(\mu_{x^\ast}^{\delta_n},\mu_{y_n^\ast}^{\delta_n})\right| + \smallO{\delta_n^3}.
\]
holds in expectation.

This is helpful because both Wasserstein metrics in the expression on the right hand side are now defined on the same space. Therefore, since $\widetilde{W}_1$ is a distance, the reverse triangle inequality implies
\[
	\left|\widetilde{W}_1(m_{x^\ast}^G,m_{y_n^\ast}^G) - \widetilde{W}_1(\mu_{x^\ast}^{\delta_n},\mu_{y_n^\ast}^{\delta_n})\right|
	\le \widetilde{W}_1(m_{x^\ast}^G, \mu_{x^\ast}^{\delta_n}) + \widetilde{W}_1(m_{y_n^\ast}^G, \mu_{y_n^\ast}^{\delta_n})
\]

Applying Proposition~\ref{prop:L1_error_adjusted_wasserstein_manifold} again we get that
\[
	\left|\widetilde{W}_1(m_{x^\ast}^G,m_{y_n^\ast}^G) - \widetilde{W}_1(\mu_{x^\ast}^{\delta_n},\mu_{y_n^\ast}^{\delta_n})\right|
	\le W_1(m_{x^\ast}^G, \mu_{x^\ast}^{\delta_n}) + W_1(m_{y_n^\ast}^G, \mu_{y_n^\ast}^{\delta_n}) + \smallO{\delta_n^3},
\]
holds in expectation, conditioned on the good event. However, the right hand side no longer involves the extended distance. Hence, it now suffices to show that for any $x \in \BallM{x^\ast}{\delta_n}$,
\begin{equation}\label{eq:proof_overview_measures}
	\Exp{W_1(m_{x}^G, \mu_{x}^{\delta_n})} = \smallO{\delta_n^3}.
\end{equation}

\subsection{Approximating probability measures on graph balls}\label{ssec:approx_graph_measures}

Recall that $\BallM{x}{\delta_n}$ denotes the closed ball around $x \in \Mcal$ with radius $\delta_n$ according to the manifold distance $d_\Mcal$. The first step in establishing~\eqref{eq:proof_overview_measures} is to move from uniform measures on the graph balls $\BallG{x}{\delta_n}$ to uniform measures on the nodes of the graph that lie in the manifold balls $\BallM{x}{\delta_n}$. The reason for this is that $y \in \BallG{x}{\delta_n}$ does not necessarily imply that $y \in \BallM{x}{\delta_n}$, nor vice versa. This creates difficulties when comparing the measures $m_{x}^G$ and $\mu_x^{\delta_n}$.

Let $G_n = \Gbb_n(x^\ast, \varepsilon_n)$ be rooted random graphs on $\Mcal$. Then we define the probability measures ${\bf m}^\Mcal$ on the nodes of $G_n$ as
\begin{equation}\label{eq:def_discrete_measure_M}
	m_x^\Mcal(y) = \begin{cases}
		\frac{1}{\left|\BallM{x}{\delta_n} \cap G_n\right|} &\mbox{if } y \in \BallM{x}{\delta_n} \cap G_n\\
		0 &\mbox{else.}
	\end{cases}
\end{equation}

Although the uniform measures $m_{x^\ast}^G$ and $m_{x^\ast}^\Mcal$ are not the necessarily equal, the Wasserstein metric between them is sufficiently small.

\begin{proposition}\label{prop:coupling_graph_discrete_rw}
Let $G_n = \Gbb_n(x^\ast, \varepsilon_n)$ be rooted random graphs on $\Mcal$ with graph distance $d_G$ that is a $\delta_n$-good approximation of $d_\Mcal$. Let $x \in \BallM{x^\ast}{\delta_n}$ and denote by $m_{x}^G$ the uniform measure on $\BallG{x}{\delta_n}$ and by $m_{x}^\Mcal$ the uniform measure on $\BallM{x}{\delta_n} \cap G_n$. Then
\[
	\Exp{W_1(m_{x}^G, m_{x}^\Mcal)} = \smallO{\delta_n^3}.
\]
\end{proposition}

The proof of this result is based on some simple computations regarding Poisson random variables and can be found in Section~\ref{ssec:proof_probability_measures_graphs}.

Proposition~\ref{prop:coupling_graph_discrete_rw} allows us to replace~\eqref{eq:proof_overview_measures} with
\begin{equation}\label{eq:proof_overview_coupling}
	\Exp{W_1(m_x^\Mcal, \mu_{x}^{\delta_n})} = \smallO{\delta_n^3}.
\end{equation}
Note that the only dependence on the graph is now in the amount of nodes placed inside the ball $\BallM{x}{\delta_n}$, which is completely determined by the Poisson process. All dependencies on the actual structure of the graph have been removed. This allows us to compute the Wasserstein metric between $m_x^\Mcal$ and $\mu_x^{\delta_n}$.

\subsection{Coupling continuous and discrete probability measures on $\Mcal$}\label{ssec:coupling_continuous_discrete_measures}

Recall that the Wasserstein metric $W_1(\mu_1,\mu_2)$ takes an infimum over all possible joint distributions (couplings) between the measures $\mu_1$ and $\mu_2$. Hence, to show that~\eqref{eq:proof_overview_coupling} holds, we need to design an optimal coupling (transport plan) between $m_x^\Mcal$ and $\mu_x^{\delta_n}$. The main idea here is to view $m_x^\Mcal$ as a discrete version of $\mu_x^{\delta_n}$. 

For now, let us assume that we are working in the $N$-dimensional Euclidean cube $\Mcal = [0,1]^N$. Given a realization of the Poisson process, a transport plan between $m_x^\Mcal$ and $\mu_x^{\delta_n}$ should assign to each measurable set $A \subseteq \BallM{x}{\delta_n}$ how much of the associated mass $\mu_x^{\delta_n}(A)$ is transported to each point of the Poisson process. To make it optimal, we should distribute the mass over those points that are closest to $A$. This problem is actually related to that of finding a minimal matching between points of a Poisson process and points of a grid on $[0,1]^N$, see~\cite{leighton1986tight,talagrand1994matching,shor1991minimax}. Here, minimal means that the largest distance between a point of the Poisson process and its matched grid point is minimized. 

The idea for the transport plan is as follows:
\begin{enumerate}
\item Place a grid on $[0,1]^N$.
\item Find a minimal matching between the Poisson process and the grid.
\item Given a $A \subseteq \BallM{x}{\delta_n}$, we take all points of the Poisson process that are matched to grid points inside $A$ and distribute the mass $\mu_x^{\delta_n}(A)$ equally over those points.
\end{enumerate}

Using known results for minimal matchings, it can then be shown that, under suitable conditions, the Wasserstein metric between $m_x^\Mcal$ and $\mu_x^{\delta_n}$ is $\smallO{\delta_n^3}$.

Finally, we need to extend these results in flat Euclidean space to the ball $\BallM{x}{Q \delta_n}$ in general $\Mcal$. For this we use that $\delta_n \to 0$ and that small neighborhoods of $x \in \Mcal$ can be mapped diffeomeorphically to the flat $N$-dimensional tangent space by the exponential map $\exp_{x} : T_{x} \Mcal \to \Mcal$. We then apply the matching results there and map back. Here we need to tread carefully, since the exponential map does not preserve distances. We thus fix a sufficiently small neighborhood $U$ around the origin of the tangent space at $x$. Then, for some fixed $0 < \xi < 1$ and large enough $n$ we have
\[
	\BallN{0}{\frac{\delta_n}{1 + \xi}} \subseteq \exp^{-1} \BallM{x}{\delta_n} 
	\subseteq \BallN{0}{\frac{\delta_n}{1 + \xi}},
\]
where $\BallN{0}{\delta}$ is the Euclidean ball of radius $\delta$. This then yields matching upper and lower bounds on the Wasserstein metric on $\Mcal$ in terms of the Wasserstein metric on the Euclidean space.

All the details of this approach are provided in Section~\ref{ssec:proof_coupling_discrete_continuous_measures}. In the end we obtain the following result.

\begin{proposition}\label{prop:coupling_uniform_discrete_rw}
For any point $x \in \Mcal$,
\[
	\Exp{W_1(m_x^\Mcal, \mu_x^{\delta_n})} = \smallO{\delta_n^3}.
\]
\end{proposition}

\subsection{Proof of the main results}

We now have all ingredients to prove the main results. We start with Theorem~\ref{thm:convergence_ollivier_graphs_general}, where we bound the expression inside the expectation as a sum of several terms and use the above results and the fact that $d_G$ is a $\delta_n$-good approximation to show that each individual term goes to zero.

\begin{proof}[Proof of Theorem~\ref{thm:convergence_ollivier_graphs_general}]
First, we bound the term inside the expectation as follows
\begin{align*}
	&\hspace{-30pt}\left|\frac{2(N + 2)\kappa(x^\ast, y_n^\ast;\, \mathcal{G}_n)}{d_G(x^\ast, y_n^\ast)^2} 
		-\Ric(\vec{v}, \vec{v})\right|\\
	&\le \left|\frac{2(N + 2)\kappa(x^\ast, y_n^\ast;\, \mathcal{G}_n)}{d_G(x^\ast, y_n^\ast)^2} 
		- \frac{2(N + 2)\kappa(x^\ast, y_n^\ast)}{\delta_n^2} \right|
		+ \left|\frac{2(N + 2)\kappa(x^\ast, y_n^\ast)}{\delta_n^2} -\Ric(\vec{v}, \vec{v})\right|.
\end{align*}
The last term is deterministic and goes to zero by Theorem~\ref{thm:convergence_curvature_manifolds}. For the first term we note that when $x^\ast$ and $y_n^\ast$ are not connected, $\kappa(x^\ast, y_n^\ast, \mathcal{G}_n) = 0$ and hence
\[
	\left|\frac{2(N + 2)\kappa(x^\ast, y_n^\ast;\, \mathcal{G}_n)}{d_G(x^\ast, y_n^\ast)^2} 
	- \frac{2(N + 2)\kappa(x^\ast, y_n^\ast)}{\delta_n^2} \right|
	\le \frac{2(2+N)}{\delta_n^{2}}.
\]
Conditioned on the good event $\Omega_n$, this does not happen by property $\Omega 3$ in Lemma~\ref{lem:good_event}, so that
\begin{align*}
	&\hspace{-30pt}\Exp{\left|\frac{2(N + 2)\kappa(x^\ast, y_n^\ast;\, \mathcal{G}_n)}{d_G(x^\ast, y_n^\ast)^2} 
		- \frac{2(N + 2)\kappa(x^\ast, y_n^\ast)}{\delta_n^2} \right|}\\
	&\le 2(N+2)\CExp{\left|\frac{\kappa(x^\ast, y_n^\ast;\, \mathcal{G}_n)}{d_G(x^\ast, y_n^\ast)^2} 
		- \frac{\kappa(x^\ast, y_n^\ast)}{\delta_n^2} \right|}{\Omega_n}
		+ \left(1 - \Prob{\Omega_n}\right) \frac{2(2+N)}{\delta_n^{2}}.	
\end{align*}

By construction of the good event we have $1- \Prob{\Omega_n} = \smallO{\delta_n^3}$ and thus, the last term in the above bound goes to zero.

For the other term we first note that since $d_G$ is a $\delta_n$-good approximation it follows that $\delta_n^2 = d_\Mcal(x^\ast,y_n^\ast)^2 \le d_G(x^\ast, y_n^\ast)^2$. Further, recall that
\[
	\kappa(x^\ast, y_n^\ast;\, \mathcal{G}_n) = 1 - \frac{W_1^G(m_{x^\ast}^G,m_{y_n^\ast}^G)}{d_G(x^\ast,y_n^\ast)}
	\quad \text{and} \quad 
	\kappa(x^\ast, y_n^\ast) = 1 - \frac{W_1(\mu_{x^\ast}^{\delta_n}, \mu_{y_n^\ast}^{\delta_n})}{\delta_n},
\] 
and that the absolute value of each curvature term can be bounded from above by $2$.

Then the expression inside the conditional expectation can be bounded as follows
\begin{align*}
	\left|\frac{\kappa(x^\ast, y_n^\ast;\, \mathcal{G}_n)}{d_G(x^\ast, y_n^\ast)^2} 
		- \frac{\kappa(x^\ast, y_n^\ast)}{\delta_n^2} \right|
	&\le \left|\kappa(x^\ast, y_n^\ast;\, \mathcal{G}_n) 
		\left(\frac{1}{d_G(x^\ast, y_n^\ast)^2} - \frac{1}{\delta_n^2}\right)\right|
		+ \frac{\left|\kappa(x^\ast, y_n^\ast;\, \mathcal{G}_n) - \kappa(x^\ast, y_n^\ast)\right|}{\delta_n^2}\\
	&\le 2 \frac{|\delta_n^2 - d_G(x^\ast, y_n^\ast)^2|}{\delta_n^4}
		+ \frac{1}{\delta_n^2}\left|\frac{W_1^G(m_{x^\ast}^G,m_{y_n^\ast}^G)}{d_G(x^\ast, y_n^\ast)}
		- \frac{W_1(\mu_{x^\ast}^{\delta_n}, \mu_{y_n^\ast}^{\delta_n})}{\delta_n}\right|\\
	&\le \frac{\left|\delta_n^2 - d_G(x^\ast, y_n^\ast)^2\right|}{\delta_n^4} 
		+ \frac{W_1^G(m_{x^\ast}^G,m_{y_n^\ast}^G)\left|\delta_n 
		- d_G(x^\ast, y_n^\ast)\right|}{\delta_n^4} \numberthis \label{eq:main_error_terms1}\\
	&\hspace{10pt}+ \frac{\left|W_1^G(m_{x^\ast}^G,m_{y_n^\ast}^G) - W_1(\mu_{x^\ast}^{\delta_n}, 
		\mu_{y_n^\ast}^{\delta_n})\right|}{\delta_n^3}. \numberthis \label{eq:main_error_terms2}
\end{align*}

Next, since $d_G$ is a $\delta_n$-good approximation, we can apply~\eqref{eq:def_distance_approximation} 
\[
	\left|\delta_n - d_G(x^\ast, y_n^\ast)\right| = \left|d_\Mcal(x^\ast,y_n^\ast) - d_G(x^\ast, y_n^\ast)\right|
	\le \delta_n \xi_n^2 + \xi_n^3 = \smallO{\delta_n^3}.
\]
Since $W_1^G(m_{x^\ast}^G,m_{y_n^\ast}^G) \le \delta_n$ it then follows that the second term in~\eqref{eq:main_error_terms1} goes to zero. For the first term we have
\begin{align*}
	\left|\delta_n^2 - d_G(x^\ast, y_n^\ast)^2\right| 
	&\le \left|\delta_n - d_G(x^\ast, y_n^\ast)\right|\left(\delta_n + d_G(x^\ast, y_n^\ast)\right)\\
	&\le \left(\delta_n\xi_n^2 + \xi_n^3\right) \left(\delta_n + \delta_n(1+\xi_n^2) + \xi_n^3\right)
		= \smallO{\delta_n^4}.
\end{align*}
which implies that this term also goes to zero.  

We are thus left with~\eqref{eq:main_error_terms2}, for which we have to show that
\[
	\CExp{\left|W_1^G(m_{x^\ast}^G,m_{y_n^\ast}^G) - W_1(\mu_{x^\ast}^{\delta_n}, \mu_{y_n^\ast}^{\delta_n})\right|}{\Omega_n} = \smallO{\delta_n^3}.
\]

We first replace $W_1(\mu_{x^\ast}^{\delta_n}, \mu_{y_n^\ast}^{\delta_n})$ with $\widetilde{W}_1(\mu_{x^\ast}^{\delta_n},\mu_{y_n^\ast}^{\delta_n})$ by invoking Proposition~\ref{prop:L1_error_adjusted_wasserstein_manifold}:
\[
	\CExp{\left|\widetilde{W}_1(\mu_{x^\ast}^{\delta_n},\mu_{y_n^\ast}^{\delta_n}) - W_1(\mu_{x^\ast}^{\delta_n},\mu_{y_n^\ast}^{\delta_n})\right|}{\Omega_n}
	 = \smallO{\delta_n^3}.
\]
This then implies
\[
	\CExp{\left|W_1^G(m_{x^\ast}^G,m_{y_n^\ast}^G) - W_1(\mu_{x^\ast}^{\delta_n}, \mu_{y_n^\ast}^{\delta_n})\right|}{\Omega_n}
	\le \CExp{\left|\widetilde{W}_1(m_{x^\ast}^G, m_{y_n^\ast}^G) 
	- \widetilde{W}_1(\mu_{x^\ast}^{\delta_n},\mu_{y_n^\ast}^{\delta_n})\right|}{\Omega_n}
	+ \smallO{\delta_n^3}.
\]

To show that the first term in the upper bound is also $\smallO{\delta_n^3}$ we apply the reverse triangle inequality twice to obtain
\begin{align*}
	\left|\widetilde{W}_1(m_{x^\ast}^G, m_{y_n^\ast}^G) - \widetilde{W}_1(\mu_{x^\ast}^{\delta_n},\mu_{y_n^\ast}^{\delta_n})\right|
	&\le \widetilde{W}_1(m_{x^\ast}^G, \mu_{x^\ast}^{\delta_n}) + \widetilde{W}_1(m_{y_n^\ast}^G, \mu_{y_n^\ast}^{\delta_n}).
\end{align*}

We proceed to show that $\widetilde{W}_1(m_{x^\ast}^G, \mu_{x^\ast}^{\delta_n}) = \smallO{\delta_n^3}$ holds in expectation on the event $\Omega_n$. The proof for $\widetilde{W}_1(m_{y_n^\ast}^G, \mu_{y_n^\ast}^{\delta_n})$ is similar.

Applying Proposition~\ref{prop:L1_error_adjusted_wasserstein_manifold} again we get
\[
	\CExp{\widetilde{W}_1(m_{x^\ast}^G, \mu_{x^\ast}^{\delta_n})}{\Omega_n} 
	\le \CExp{W_1(m_{x^\ast}^G, \mu_{x^\ast}^{\delta_n})}{\Omega_n} + \smallO{\delta_n^3}
	\le \bigO{1} \Exp{W_1(m_{x^\ast}^G, \mu_{x^\ast}^{\delta_n})} + \smallO{\delta_n^3}
\]
Since the expectation is $\smallO{\delta_n^3}$ by Proposition~\ref{prop:coupling_uniform_discrete_rw} we conclude that
\[
	\CExp{\left|W_1^G(m_{x^\ast}^G,m_{y_n^\ast}^G) - W_1(\mu_{x^\ast}^{\delta_n}, \mu_{y_n^\ast}^{\delta_n})\right|}{\Omega_n} = \smallO{\delta_n^3},
\]
which finishes the proof.
\end{proof}

Now that we have the general result, Theorem~\ref{thm:convergence_ollivier_almost_sparse_graphs_weighted} and Theorem~\ref{thm:convergence_ollivier_graphs_hopcount} directly follow from Theorem~\ref{thm:convergence_ollivier_graphs_general} if we can show that the graph distances that are considered there are $\delta_n$-good approximations. 

Throughout the remainder of this section we will assume that
\begin{align*}
	\varepsilon_n &= \bigT{\log(n)^a n^{-\alpha}}\\
	\delta_n &= \bigT{\log(n)^b n^{-\beta}},
\end{align*}
for some $a,b \in \R$ and $0 \le \alpha, \beta \le 1$. We shall also assume that $\varepsilon_n \le \delta_n$. The following results show that for appropriate choices of the constants $a.b$ and $\alpha,\beta$ both the weighted manifold and the rescaled hopcount distance are $\delta_n$-good approximations. The proofs are given in Section~\ref{sec:weighted_graph_distances} and Section~\ref{sec:almost_sparse_graphs}, respectively.

\begin{proposition}\label{prop:weighted_graph_distance_delta_approx}
Suppose the constants in $\varepsilon_n$ and $\delta_n$ satisfy
\[
	0 < \beta \le \alpha, \quad \alpha + 2\beta \le \frac{1}{N}
\]
with $a \le b$ if $\alpha = \beta$ and $a+2b > \frac{2}{N}$ if $\alpha + 2\beta = \frac{1}{N}$.
Let $y_n^\ast \in \Mcal$ be at distance $\delta_n$ in the direction of $\vec{v}$ and $G_n = \Gbb_n(x^\ast, y_n^\ast, \varepsilon_n)$ be rooted random graphs on $\Mcal$. Then the manifold-weighted graph distance $d_G^w$ on $G_n$ is a $\delta_n$-good approximation of $d_\Mcal$.
\end{proposition}

\begin{proposition}\label{prop:shortest_graph_distance_delta_approx}
Suppose the constants in $\varepsilon_n$ and $\delta_n$ satisfy
\[
	0 < \beta \le 1/9 \quad \text{and} \quad 3\beta \le \alpha \le \frac{1-3\beta}{2},
\]
and $a < 3b$ if $\alpha = 3\beta$ and $2a + 3b > 1$ if $\alpha = \frac{1-3\beta}{2}$. 
Let $y_n^\ast \in \Mcal$ be at distance $\delta_n$ in the direction of $\vec{v}$. Let $G_n = \Gbb_n(x^\ast, y_n^\ast, \varepsilon_n)$ be rooted random graphs on a $2$-dimensional Riemannian manifold $\Mcal$ and denote by $d_G^s$ the shortest path distance. Then the $\varepsilon_n$-weighted graph distance $d_G^\ast := \varepsilon_n d_G^s$ on $G_n$ is a $\delta_n$-good approximation of $d_\Mcal$.
\end{proposition}

Observe that the conditions of the constants in Proposition~\ref{prop:weighted_graph_distance_delta_approx} and Proposition~\ref{prop:shortest_graph_distance_delta_approx} are exactly the same as in Theorem~\ref{thm:convergence_ollivier_almost_sparse_graphs_weighted} and Theorem~\ref{thm:convergence_ollivier_graphs_hopcount}, respectively. Moreover, these conditions imply that $\lambda_n = \smallO{\varepsilon_n}$ and $\lambda_n = \smallO{\delta_n^3}$, with $\lambda_n$ as defined in~\eqref{eq:def_lambda_n}, as we will now demonstrate. 

In Proposition~\ref{prop:weighted_graph_distance_delta_approx} we have $\beta > 0$ and $\alpha + 2\beta \le \frac{1}{N}$. It then follows that $\alpha < \frac{1}{N}$ which implies $\lambda_n = \smallO{\varepsilon_n}$. When the inequality $3\beta \le \alpha + 2\beta \le \frac{1}{N}$ is strict we have that $\lambda_n = \smallO{\delta_n^3}$. When $3\beta = \frac{1}{N}$ it must be that $\alpha + 2\beta = \frac{1}{N}$ and hence the conditions of Proposition~\ref{prop:weighted_graph_distance_delta_approx} imply that $3b \ge a + 2b > \frac{2}{N}$. From this we deduce that $\lambda_n/\delta_n^3 = \bigT{\log(n)^{\frac{2}{N} - a - 2b}} = \smallO{1}$.

In Proposition~\ref{prop:shortest_graph_distance_delta_approx}, since $N = 2$, the conditions $\lambda_n = \smallO{\varepsilon_n}$ and $\lambda_n = \smallO{\delta_n^3}$ follow if $\alpha < \frac{1}{2}$ and $3\beta < \frac{1}{2}$. The first inequality holds since $\beta > 0$ and $\alpha \le \frac{1 - 3\beta}{2}$, while the second is due to the fact that $3\beta \le \frac{3}{9} = \frac{1}{3}$.

We thus conclude that under the conditions in both propositions, the radii satisfy the conditions of Theorem~\ref{thm:convergence_ollivier_graphs_general}. Hence, Theorem~\ref{thm:convergence_ollivier_almost_sparse_graphs_weighted} and Theorem~\ref{thm:convergence_ollivier_graphs_hopcount} follow from it.

\section{Proofs}\label{sec:proofs}

Here we prove all the intermediate results that we used to prove our main results in the previous section. We start with the proof of Lemma~\ref{lem:good_event} and Proposition~\ref{prop:L1_error_adjusted_wasserstein_manifold} in the next Section~\ref{ssec:extending_graph_distance}. In Section~\ref{ssec:proof_probability_measures_graphs} we provide the details for Proposition~\ref{prop:coupling_graph_discrete_rw}, while the proof of Proposition~\ref{prop:coupling_uniform_discrete_rw} is given in Section~\ref{ssec:proof_coupling_discrete_continuous_measures}. We end with Sections~\ref{sec:weighted_graph_distances} and~\ref{sec:almost_sparse_graphs} where we prove Propositions~\ref{prop:weighted_graph_distance_delta_approx} and~\ref{prop:shortest_graph_distance_delta_approx}, respectively, leading to the main results of this paper.

Recall that
\[
	\lambda_n = \log(n)^{\frac{2}{N}} n^{-\frac{1}{N}},
\]
and $\varepsilon_n \le \delta_n \to 0$ are such that $\lambda_n = \smallO{\varepsilon_n}$ and $\lambda_n = \smallO{\delta_n^3}$.

\subsection{Extended graph distance}\label{ssec:extending_graph_distance}

Our first goal is to proof Lemma~\ref{lem:good_event}. We start by showing that for sufficiently small radius $r_n$ and any finite set of points $u \in \Mcal$, the balls $\BallM{u}{r_n}$ will each contain at least one node from the rooted graphs $G_n = \mathbb{G}_n(x^\ast, y_n^\ast, \varepsilon_n)$.

\begin{lemma}\label{lem:non_empty_varepsilon_balls}
Let $U \subset \Mcal$ be a finite set of points in $\Mcal$ such that $|U| = \bigO{n^c}$, for some $c > 0$, and let $r_n = \bigT{\lambda_n}$. Then, for $G_n = \mathbb{G}_n(\varepsilon_n)$,
\[
	\Prob{\bigcup_{u \in U} \{|\BallM{u}{r_n} \cap G_n| = 0\}} = \smallO{\delta_n^3},
\]
as $n \to \infty$.
\end{lemma}

\begin{proof}
First note that for $r_n$ small enough the ball $\BallM{u}{r_n}$ can be mapped diffeomorphically onto the tangent space $T_u \Mcal$ at $u$. In particular, for small enough $r_n$ we have that, as $n \to \infty$, $\vol_\Mcal\left(\BallM{u}{r_n}\right) = \bigT{r_n^N} = \bigT{\lambda_n^N}$. Next, since the nodes in $G_n$ are placed according to a Poisson process with intensity $n/\vol_\Mcal(\Mcal)$ it follows that
\[
	\Prob{|\BallM{u}{r_n} \cap G_n| = 0} = e^{-\frac{n \vol_\Mcal\left(\BallM{u}{\lambda_n}\right)}{\vol_\Mcal(\Mcal)}}
	= e^{-n \bigT{\lambda_n^N}} = e^{-\bigT{\log(n)^{2}}}.
\]
Therefore, by applying the union bound we get
\begin{align*}
	\Prob{\bigcup_{u \in U} \{|\BallM{u}{\varepsilon_n} \cap G| = 0\}}
	&\le |U| \, \Prob{|\BallM{u}{\varepsilon_n} \cap G| = 0}\\
	&= e^{-\bigT{\log(n)^{2}} + \log(|U|)}
		\le e^{-\bigT{\log(n)^{2}} + c \log(n)}.
\end{align*}
To finish the proof we note that $e^{-\bigT{\log(n)^{2}} + c \log(n)} = \smallO{\lambda_n}$ which by assumption is $\smallO{\delta_n^3}$.
\end{proof}

With this lemma we obtain the following corollary.

\begin{corollary}\label{cor:covering}
There exists a collection $\{B_1, \dots, B_{m}\}$ of $m = \bigT{\lambda_n^{-N}}$ balls of radius $\lambda_n/4$ that cover $\Mcal$, such that if we denote by $c_1, \dots, c_m$ their centers and define the event
\begin{equation}\label{eq:def_cover_event}
	C_n = \bigcap_{t = 1}^m \left\{\left|\BallM{c_t}{\lambda_n/4} \cap G_n\right| \ne 0\right\},
\end{equation}
then
\[
	\Prob{C_n} = 1 - \smallO{\delta_n^3}.
\]
\end{corollary}

\begin{proof}
The collection is constructed using the standard trick of taking a maximal set of \emph{disjoint} balls of radius $\lambda_n/8$ in $\Mcal$. Denote their centers by $c_1, \dots, c_m$. Simple volume comparison, and the compactness of $\Mcal$, gives $m = \bigO{\lambda_n^{-N}}$. By construction, the balls $B_i =  \BallM{c_i}{\lambda_n/4}$ then cover $\Mcal$, and hence $m= \bigT{\lambda_n^{-N}} = \bigT{\log(n)^{-2} n} = \bigO{n}$. The result then follows from Lemma~\ref{lem:non_empty_varepsilon_balls}.
\end{proof}

The event $C_n$ will play a crucial part in defining the \emph{good event} $\Omega_n$. Let $D_n$ denote the event on which~\eqref{eq:def_distance_approximation} holds. Then we define the \emph{good event} as 
\begin{equation}\label{eq:def_good_event}
	\Omega_n = C_n \cap D_n.
\end{equation} 

On this event, with sufficiently high probability, $(\BallM{x^\ast}{Q\delta_n}, \widetilde{d}_\Mcal)$ is a metric space for any constant $Q > 0$ and the extended distance $\widetilde{d}_\Mcal$ is a good approximation of the original distance $d_\Mcal$. Note that we do not need to consider the whole manifold since curvature is a local property.

\begin{lemma}\label{lem:extended_distance_metric_space}
Let $\Omega_n$ be the event defined in~\eqref{eq:def_good_event} and $Q > 3$ the constant from Definition~\ref{def:delta_good_approx}. Then on the event $\Omega_n$, 
\begin{enumerate}
\item each pair of points $u,v \in \BallM{x^\ast}{Q \delta_n}$ is connected by a path in the extended graph $G_n(u,v)$ and
\item $(\BallM{x^\ast}{Q\delta_n}, \widetilde{d}_\Mcal)$ is a metric space.
\end{enumerate}
\end{lemma}

\begin{proof}

\begin{figure}
\centering
\begin{tikzpicture}[scale=0.8]
	\tikzstyle{vertex}=[fill, circle, inner sep=0pt, minimum size=5pt]
	\tikzstyle{edge}=[color=black]
	
	\draw node[vertex] (u) at (0,0) {};
	\path (u)+(40:0.5) node[vertex,red] (c1) {};
	\draw node (z1) at (3,0) {};
	\path (z1)+(-90:0.5) node[vertex,red] (c2) {};
	\draw node (z2) at (6,0) {};
	\path (z2)+(120:0.8) node[vertex,red] (c3) {};
	\draw node (z4) at (12,0) {};
	\path (z4)+(80:0.5) node[vertex,red] (c4) {};
	\draw node[vertex] (v) at (15,0) {};
	\path (v)+(-110:0.6) node[vertex,red] (c5) {};
	
	\path (u)+(100:0.8) node[vertex,blue] (x0) {};
	\path (z1)+(210:0.8) node[vertex,blue] (x1) {}; 
	\path (z2)+(160:0.4) node[vertex,blue] (x2) {}; 
	\path (x2)+(15:2.5) node (x31) {};
	
	\path (z4)+(300:0.9) node[vertex,blue] (x4) {};
	\path (x4)+(150:2.5) node (x32) {};
	\path (v)+(50:0.6) node[vertex,blue] (x5) {}; 
	
	\path (u)+(180:0.5) node {$u$};
	\path (v)+(0:0.5) node {$v$};
	
	\path (c1)+(45:0.5) node {\color{red} $c_{t_1}$};
	\path (c2)+(-45:0.5) node {\color{red} $c_{t_2}$};
	\path (c3)+(90:0.5) node {\color{red} $c_{t_3}$};
	\path (c4)+(90:0.5) node {\color{red} $c_{t_{k-1}}$};
	\path (c5)+(270:0.5) node {\color{red} $c_{t_k}$};

	\path (x0)+(200:0.4) node {\color{blue} $x_{t_1}$};
	\path (x1)+(-90:0.4) node {\color{blue} $x_{t_2}$};
	\path (x2)+(-90:0.4) node {\color{blue} $x_{t_3}$};
	\path (x4)+(300:0.5) node {\color{blue} $x_{t_{k-1}}$};
	\path (x5)+(45:0.5) node {\color{blue} $x_{t_k}$};
	
	\draw[thick] (u) -- (z2);
	\draw node at (9,0) {$\dots\dots$};
	\draw[thick] (z4) -- (v);
	\draw[blue] (u) -- (x0) -- (x1) -- (x2) -- (x31);
	\draw[blue] (x32) -- (x4) -- (x5) -- (v);
	
	\draw[dashed,red] (c1) circle  (1.5cm);
	\draw[dashed,red] (c2) circle  (1.5cm);
	\draw[dashed,red] (c3) circle  (1.5cm);
	\draw[dashed,red] (c4) circle  (1.5cm);
	\draw[dashed,red] (c5) circle  (1.5cm);
	
\end{tikzpicture}
\caption{Depiction of the covering of the geodesic between $u$ and $v$ by the balls $B_{t_i}$.}
\label{fig:cover_geodesic_uv}
\end{figure}
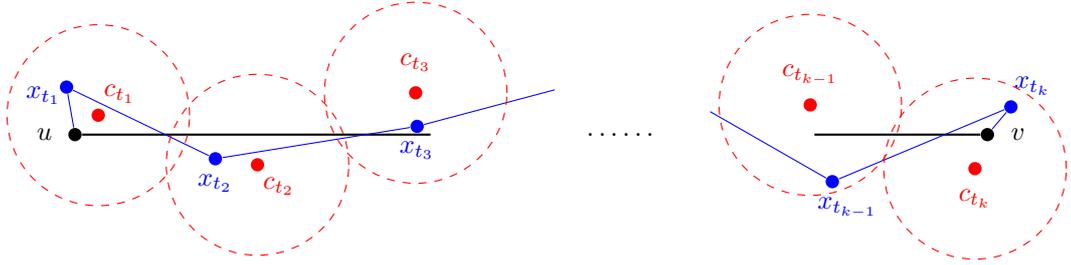

We first prove the first statement. For this, take any $u,v \in \BallM{x^\ast}{Q\delta_n}$ and let $\gamma(u,v)$ denote the geodesic between $u$ and $v$. This geodesic will be covered by a subsequence $B_{t_1}, \dots, B_{t_k}$ of the cover of $\Mcal$, which we rank in order of appearance moving from $u$ to $v$. Let $c_{t_1}, \dots, c_{t_k}$ denote the corresponding centers of these balls, see Figure~\ref{fig:cover_geodesic_uv}. On the event $C_n$ each ball contains a vertex $x_{t_i} \in G_n$ and since 
\[
	d_\Mcal(u,x_{t_1}), d_\Mcal(v,x_{t_k}) \le 2 \frac{\lambda_n}{4} = \frac{\lambda_n}{2}
\]
the edges $(u, x_{t_1})$ and $(v,x_{t_k})$ are present in $G_n(u,v)$. Moreover, since $d_\Mcal(x_{t_i},x_{t_{i+1}})$ is bounded by four times the radius of the balls, it follows that for large enough $n$, $d_\Mcal(x_{t_i},x_{t_{i+1}}) \le \lambda_n = \smallO{\varepsilon_n}$ and thus, for $n$ large enough, $\{x_{t_1},\dots, x_{t_k}\}$ is a path in $G_n$. We thus conclude that $u$ and $v$ are connected in $G_n(u,v)$.

Note that because of this property, on the event $\Omega_n$, the extended manifold distance between $\widetilde{d}_\Mcal$ is well-defined on $\Mcal$. 

We are left to show that on the event $\Omega_n$, the extended manifold distance is a true distance. Note that the only non-trivial part is the triangle inequality. Let $u,v,z\in \BallM{x^\ast}{Q\delta_n}$ and consider the graphs $G^{(1)} = G_n(u,v)$ and $G^{(2)} = G_n(u,v,z)$. Now observe that the triangle inequality can only be violated if $z$ creates a short-cut, i.e. if the shortest weighted path between $u$ and $v$ in $G^{(1)}$ is longer than in $G^{(2)}$. Suppose that this is true, and let $\pi_1 = \{u, \dots, y_1, z, y_2, \dots, v\}$ denote this new weighted shortest path in $G^{(2)}$. Since $y_1$ and $y_2$ are connected to $z$ in $G^{(2)}$ it follows that $d_\Mcal(z,y_i) \le \lambda_n/2$. However, by the triangle inequality for $d_\Mcal$, this implies that $d_\Mcal(y_1,y_2) \le \lambda_n = \smallO{\varepsilon_n}$ and hence, for sufficiently large $n$, the edge $(y_1, y_2)$ is present in $G_n$ and thus also in $G^{(1)}$ and $G^{(2)}$. 

Let $\hat{\pi} = \{y_1 := x_0, x_{1}, \dots, x_{m-1}, y_2:=x_m\}$ denote the shortest weighted path in $G_n$ between $y_1$ and $y_2$, i.e. $d_G(y_1,y_2) = \sum_{t = 1}^m w_{x_{t-1} \, x_{t}}$, and take $\pi_2 = \{u, \dots, y_1, x_1, \dots, x_{m-1}, y_2, \dots, v\}$. Then $\pi_2$ is a path between $u$ and $v$ that excludes $z$. See also Figure~\ref{fig:triangle_inequality_proof}. We will show that the total weight of this path is at most that of $\pi_1$.

For simplicity lets us denote by $\|\pi\|$ the total weight of a path $\pi$. Since $d_G$ is a $\delta_n$-good approximation,
\[
	\|\hat{\pi}\| := \sum_{t = 1}^m w_{x_{t-1} \, x_{t}} = d_G(y_1,y_2) \le d_{\Mcal}(y_1,y_2)(1+\xi_n^2) + \xi_n^3
\]
holds on the event $\Omega_n$. Applying the triangle inequality for $d_\Mcal$ we get
\begin{align*}
	\|\hat{\pi}\| &\le d_\Mcal(y_1,z)(1+\xi_n^2) + d_\Mcal(z,y_2)(1+\xi_n^2) + \xi_n^3\\
	&\le d_\Mcal(y_1,z)(1+\xi_n^2) + d_\Mcal(z,y_2)(1+\xi_n^2) + 2\xi_n^3 = w_{y_1 z} + w_{y_2 z}. 
\end{align*}
This implies that the total weight of the path $\pi_2$ is at most that of $\pi_1$ from which we conclude that $z$ cannot create a short-cut and hence $\widetilde{d}_\Mcal$ satisfies the triangle inequality.

\begin{figure}
\centering
\begin{tikzpicture}
	\tikzstyle{vertex}=[fill, circle, inner sep=0pt, minimum size=7pt]
	\draw node[vertex] (u) at (0,0) {};
	\draw node[vertex] (v) at (10,0) {};
	
	\path (u)+(3.5,-2) node[vertex] (y1) {};
	\path (y1)+(1.5,-0.5) node[vertex] (w) {};
	\path (y1)+(3,0) node[vertex] (y2) {};
	
	\path (u)+(135:0.5) node {$u$};
	\path (v)+(45:0.5) node {$v$};
	\path (y1)+(-135:0.5) node {$y_1$};
	\path (w)+(-90:0.5) node {$z$};
	\path (y2)+(-45:0.5) node {$y_2$};
	\path (y1)+(1.5,1) node {$\hat{\pi}$};
	
	\draw[thick,decorate, decoration=snake] (u) -- (v);
	\draw[thick,blue,decorate, decoration=snake] (u) -- (y1);
	\draw[thick,dashed] (y1) -- (w);
	\draw[thick,dashed] (w) -- (y2);
	\draw[thick,blue,decorate, decoration=snake] (y2) -- (v);
	\draw[thick,red] (y1) -- (y2);
	\draw[thick,blue,decorate, decoration=snake] (y1) to[out=45,in=135] (y2);
\end{tikzpicture}
\caption{Abstract depiction of the weighted shortest path between $u$ and $v$ created by adding $z$ and the path $\pi_2$, given in blue.}
\label{fig:triangle_inequality_proof}
\end{figure}
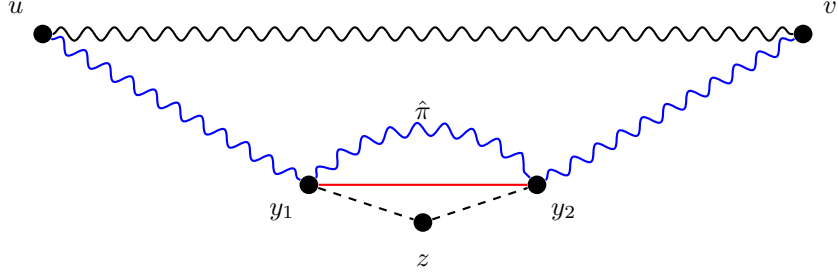
\end{proof}


We are now ready to prove Lemma~\ref{lem:good_event}. 

\begin{proof}[Proof of Lemma~\ref{lem:good_event}]

Note that for any two nodes $u,v \in G_n$ with $u,v \in \BallM{x^\ast}{Q \delta_n}$, Lemma~\ref{lem:extended_distance_metric_space} implies that $u$ and $v$ are connected by a path in $G_n$. Hence the only part of Lemma~\ref{lem:good_event} to prove is property $\Omega 2$ there.

Take any  $u,v \in \BallM{x^\ast}{3\delta_n}$. Then on the event $\Omega_n$, by definition of the extended distance $\widetilde{d}_\Mcal$, there exists $x_u, x_v \in G_n$ such that $d_\Mcal(u,x_u) \le \lambda_n/2$, $d_\Mcal(v,x_v) \le \lambda_n/2$ and
\begin{align*}
	\widetilde{d}_\Mcal(u,v) &= d_\Mcal(u,x_u)(1 + \xi_n^2) + d_\Mcal(v,x_v)(1 + \xi_n^2) + 2\xi_n^3 + d_G(x_u,x_v)\\
	&\le \lambda_n (1 + \xi_n^2) + 2\xi_n^3 + d_G(x_u,x_v). \numberthis \label{eq:proof_good_event_d1}
\end{align*}
Moreover, since $Q > 3$ and $\lambda_n = \smallO{\delta_n^3}$ we can assume that $x_u, x_v \in \BallM{x^\ast}{Q\delta_n}$, for sufficiently large $n$. Since the approximation~\eqref{eq:def_distance_approximation} holds on the event $\Omega_n$, we have
\begin{align*}
	\left|d_G(x_u,x_v) - d_\Mcal(u,v)\right| &\le \left|d_G(x_u,x_v) - d_\Mcal(x_u,x_v)\right|
			+ \left|d_\Mcal(x_u,x_v) - d_\Mcal(u,v)\right| \\
		&\le \left|d_G(x_u,x_v) - d_\Mcal(x_u,x_v)\right| + d_\Mcal(x_u,u) + d_\Mcal(x_v,v) \\
		&\le d_\Mcal(x_u,x_v)\xi_n^2 + \xi_n^3 + d_\Mcal(x_u,u) + d_\Mcal(x_v,v) \\
		&\le d_\Mcal(x_u,x_v)\xi_n^2 + \xi_n^3 + \lambda_n.	\numberthis \label{eq:proof_good_event_d2}
\end{align*}

Combining~\eqref{eq:proof_good_event_d1} and~\eqref{eq:proof_good_event_d2} we get
\begin{align*}
	\left|\widetilde{d}_\Mcal(u,v) - d_\Mcal(u,v)\right| 
	&\le \left|\widetilde{d}_\Mcal(u,v) - d_G(x_u,x_v)\right| 
		+ \left|\vphantom{\widetilde{d}_\Mcal}d_G(x_u,x_v) - d_\Mcal(u,v)\right| \\
	&\le 2\widetilde{d}_\Mcal(u,x_u) + 2\widetilde{d}_\Mcal(v,x_v) 
		+ \left|d_G(x_u,x_v) - d_\Mcal(x_u,x_v)\right|\\
	&\le \lambda_n(1+\xi_n^2) + d_\Mcal(x_u,x_v)\xi_n^2 + 3\xi_n^3 + \lambda_n
\end{align*}
Applying the triangle inequality to the last distance, 
\[
	d_\Mcal(x_u,x_v) \le d_\Mcal(u,v) + d_\Mcal(u,x_u) + d_\Mcal(v,x_v) \le d_\Mcal(u,v) + \lambda_n
\]
we get
\[
	\left|\widetilde{d}_\Mcal(u,v) - d_\Mcal(u,v)\right| \le d_\Mcal(u,v)\xi_n^2 + 2\lambda_n(1 + \xi_n^2) + 3\xi_n^3
	= \smallO{\delta_n^3}.
\]

\end{proof}

%
%

Finally, we need to prove Proposition~\ref{prop:L1_error_adjusted_wasserstein_manifold}. Since, on the event $\Omega_n$, we have
\[
	\left|\widetilde{d}_\Mcal(u,v) - d_\Mcal(u,v)\right| \le \smallO{\delta_n^3}.
\] 
the proof follows immediately from the following elementary result on Wasserstein metrics.

\begin{lemma}\label{lem:wasserstein_metric_approx_general}
Let $(\Xcal,d)$ and $(\Xcal, \widetilde{d})$ be two metric spaces and $U \subseteq \Xcal$ such that
\[
	|d(x,y) - \widetilde{d}(x,y)| \le K
\]
holds for all $x,y \in U$ and some $K > 0$. Denote by $W_1$ and $\widetilde{W}_1$ the Wasserstein metric associated with $d$ and $\widetilde{d}$, respectively. Then for any two probability measures $\mu_1$ and $\mu_2$ on $U$,
\[
	\left|\widetilde{W}_1(\mu_1, \mu_2) - W_1(\mu_1, \mu_2)\right| \le K.
\]
\end{lemma}

\begin{proof}
For any coupling $\mu$ between $\mu_1$ and $\mu_2$,
\begin{align*}
	\int \widetilde{d}(x,y) \dd \mu(x,y)
	&\le \int \left(d(x,y) + K\right) \dd \mu(u,v) 
		\le \int d(x,y) \dd \mu(x,y) + K.
\end{align*}
and similarly
\[
	\int d(x,y) \dd \mu(x,y) \ge \int \widetilde{d}(x,y) \dd \mu(x,y) + K.
\]
Next we note that the Wasserstein metric is achieved by some optimal coupling. Let $\mu^\ast$ denote the optimal coupling for $\mu_1$ and $\mu_2$ with respect to $d$, i.e. $W_1(\mu_1,\mu_2) = \int d(x,y) \dd \mu^\ast(x,y)$, and define $\widetilde{\mu}^\ast$ similarly. Then
\[
	\widetilde{W}_1(\mu_1,\mu_2) \le \int \widetilde{d}(x,y) \dd \mu^\ast(x,y) \le W_1(\mu_1,\mu_2) + K,
\]
and
\[
	W_1(\mu_1,\mu_2) \le \int d(x,y) \dd \widetilde{\mu}^\ast(x,y) \le \widetilde{W}_1(\mu_1,\mu_2) + K,
\]
from which the result follows.
\end{proof}

\subsection{Probability measures on graphs}\label{ssec:proof_probability_measures_graphs}

In this section we give the proof of Proposition~\ref{prop:coupling_graph_discrete_rw}. Recall that $m_x^G$ and $m_x^\Mcal$ denote the uniform probability measures on the set of nodes in $\BallG{x}{\delta_n}$ and $\BallM{x}{\delta_n}$, respectively. The goal is then to show that
\[
	\Exp{W_1(m_x^G, m_x^\Mcal)} = \smallO{\delta_n^3}.
\]

As we mentioned, these two sets are not necessarily contained in each other. Hence, to bound the Wasserstein metric we will work with slightly smaller and larger balls $B^-$ and $B^+$ such that
\[
	B^- \cap G_n \subseteq \BallG{x}{\delta_n}, \BallM{x}{\delta_n} \cap G_n \subseteq B^+ \cap G_n.
\]
We can then obtain an upper bound by comparing the Wasserstein metric between $m_x^G$, $m_x^\Mcal$ and the uniform probability measure on $B^+ \cap G_n$. This bound can be made $\smallO{\delta_n^3}$, by carefully selecting the radii of $B^-$ and $B^+$.

Before we give the details, we need the following general result concerning Poisson random variables.

\begin{lemma}\label{lem:technicall_poisson}
Let $\alpha_n, \beta_n \to \infty$ and $X_n$, $Y_n$ be two independent Poisson random variables with means $\alpha_n$ and $\beta_n$, respectively. Then
\[
	\CExp{\frac{X_n}{X_n + Y_n}}{X_n + Y_n \ge 1} = \bigO{\frac{\alpha_n}{\alpha_n + \beta_n}}.
\]
\end{lemma}

\begin{proof}
First, let $C > \sqrt{2}$ be some large fixed constant. Then we have that (c.f. \cite[Lemma 2.1]{penrose2003random})
\[
	\Prob{\left|X_n - \alpha_n\right| > C \sqrt{\alpha_n \log(\alpha_n)}}
	= \bigO{\alpha_n^{-\frac{C^2}{2}}}.
\]
In particular, if we define $\alpha_n^\pm = \alpha_n \pm C \sqrt{\alpha_n \log(\alpha_n)}$, then
\[
	\max\left\{\Prob{X_n < \alpha_n^-}, \Prob{X_n > \alpha_n^+}\right\} = \bigO{\alpha_n^{-\frac{C^2}{2}}}.
\]
Similar results hold for $Y_n$ with $\beta_n^\pm$ defined similarly.

We start by conditioning on $X_n$:
\begin{align*}
	\CExp{\frac{X_n}{X_n + Y_n}}{X_n + Y_n \ge 1}
	&= \sum_{k = 0} \CExp{\frac{k}{k + Y_n}}{Y_n \ge 1} \Prob{X_n = k}\\
	&= \sum_{k < \alpha_n^-} \CExp{\frac{k}{k + Y_n}}{Y_n \ge 1} \Prob{X_n = k}\\
	&\hspace{10pt}+ \sum_{k \ge \alpha_n^-} \CExp{\frac{k}{k + Y_n}}{Y_n \ge 1} 
		\Prob{X_n = k}\\
	&:= I_n^{(1)} + I_n^{(2)}
\end{align*}
We will bound each term separately.

First we bound the expectation inside each summation by further conditioning on $Y_n$:
\begin{align*}
	\CExp{\frac{k}{k + Y_n}}{Y_n \ge 1}
	&\le \frac{k}{k + 1} \Prob{1 \le Y_n < \beta_n^-}\\
	&\hspace{10pt}+ \sum_{\beta_n^- \le y \le \beta_n^+} 
		\frac{k}{k + y} \Prob{Y_n = y}\\
	&\hspace{10pt}+ \frac{k}{k + \beta_n^+} \Prob{Y_n > \beta_n^+}\\
	&\le k \left(\frac{1}{k+\beta_n^-} 
		+ \frac{\Prob{\left|Y_n - \beta_n\right| > C\sqrt{\beta_n \log(\beta_n)}}}{k + 1}\right)\\
	&\le \frac{k}{k + \beta_n^-} 
		\left(1 + \bigO{\beta_n^{1-\frac{C^2}{2}}}\right)
		= (1 + \smallO{1}) \frac{k}{k + \beta_n^-},
\end{align*}
because $C > \sqrt{2}$. We can now bound $I_n^{(1)}$ as follows
\begin{align*}
	I_n^{(1)} &\le \frac{\alpha_n^-}{\beta^-}\Prob{X_n < \alpha_n^-} 
	= \bigO{(\beta_n^{-})^{-1} \alpha_n^{1 - \frac{C^2}{2}}}
	= \bigO{\beta_n^{-1} \alpha_n^{1 - \frac{C^2}{2}}}
\end{align*}
where we used that $\beta_n^- \sim \beta_n$, i.e $\beta_n^-/\beta_n \to 1$.

For $I_n^{(2)}$ we have, using that $\alpha_n^- \sim \alpha_n$,
\begin{align*}
	I_n^{(2)} &\le (1+\smallO{1}) \sum_{k \ge \alpha_n^-} \frac{k}{k + \beta_n^-} \Prob{X_n = k}
		\le \bigO{\frac{\Exp{X_n}}{\alpha_n^- + \beta_n^-}}
		= \bigO{\frac{\alpha_n}{\alpha_n + \beta_n}}
\end{align*}
and thus the result follows since we are free to select $C > \sqrt{2}$ large enough so that $I_n^{(1)}$ is of smaller order.
\end{proof}

We are now ready to prove Proposition~\ref{prop:coupling_graph_discrete_rw}

\begin{proof}[Proof of Proposition~\ref{prop:coupling_graph_discrete_rw}]
Let $\delta_n^\pm = (\delta_n \pm \xi_n^3)/(1 \mp \xi_n^2)$ and let $D_n$ be the event on which approximation~\eqref{eq:def_distance_approximation} of Definition~\ref{def:delta_good_approx} holds. Then
\begin{align*}
	\Exp{W_1(m_{x}^G, m_{x}^\Mcal}
	&\le \Exp{W_1(m_{x}^G, m_{x}^\Mcal) \ind{D_n}}
		+ \bigO{\delta_n (1 - \Prob{D_n})}\\
	&= \Exp{W_1(m_{x}^G, m_{x}^\Mcal) \ind{D_n}} + \smallO{\delta_n^3}.
\end{align*}
It is thus enough to show that the first term is $\smallO{\delta_n^3}$.

Note that on the event $D_n$,
\[
	\BallM{x}{\delta_n^-} \cap G_n \subseteq \BallG{x}{\delta_n}, \BallM{x}{\delta_n} \cap G_n \subseteq 
		\BallM{x}{\delta_n^+} \cap G_n.
\]

Let $V_n \subseteq \Mcal$ be any neighborhood of $x$ such that $\vol_\Mcal(\Bcal_n) = \bigT{\delta_n^N}$ and
\[
	\BallM{x}{\delta_n^-} \subseteq \Bcal_n \subseteq 
			\BallM{x}{\delta_n^+} \cap G_n,
\]
where $\Bcal_n = V_n \cap G_n$. Denote by $m_{n}$ the uniform probability measure on $\Bcal_n$. We will prove that
\begin{equation}\label{eq:coupling_discrete_rw_general}
	\Exp{W_1(m_n,m_x^+)\ind{D_n}} = \smallO{\delta_n^3}.
\end{equation}

Since
\[
	\Exp{W_1(m_{x}^G, m_{x}^\Mcal) \ind{D_n}} \le \Exp{W_1(m_{x}^G,m_x^+)\ind{D_n}} + \Exp{W_1(m_{x}^\Mcal,m_x^+)\ind{D_n}},
\]
applying~\eqref{eq:coupling_discrete_rw_general} twice, once with $\Bcal_n = \BallG{x}{\delta_n}$ and once with $\Bcal_n =\BallM{x}{\delta_n} \cap G_n$, will yield the required result.

Let us write $\Bcal_n^\pm := \BallM{x}{\delta_n^\pm} \cap G_n$ and denote by $m_x^\pm$ the uniform probability measure on $\Bcal_n^\pm$. To establish~\eqref{eq:coupling_discrete_rw_general} we will show that
\begin{equation}\label{eq:proof_coupling_graph_discrete_measures_main}
	\Exp{W_1(m_n,m_x^+)\ind{D_n}} = \bigO{\frac{(\delta_n^+)^N - (\delta_n^-)^N}{(\delta_n^+)^{N-1}}}.
\end{equation}

Note that by definition of $\delta_n^\pm$ we have $(\delta_n^+)^N - (\delta_n^-)^N = \bigO{\xi_n^2 \delta_n^N}$. Therefore, if~\eqref{eq:proof_coupling_graph_discrete_measures_main} holds,
\begin{align*}
	\Exp{W_1(m_n, m_x^+)\ind{D_n}} &\le \bigO{\frac{(\delta_n^+)^N - (\delta_n^-)^N}{\delta_n^{N-1}}}
		= \bigO{\delta_n \xi_n^2} = \smallO{\delta_n^3},
\end{align*}
since $\xi_n = \smallO{\delta_n}$. 

To establish~\eqref{eq:proof_coupling_graph_discrete_measures_main} we condition on $|\Bcal_n^-|$
\begin{align*}
	\Exp{W_1(m_n, m_{x}^+) \ind{D_n}} 
	&= \CExp{W_1(m_n,m_x^+)\ind{D_n}}{|\Bcal_n^-| = 0} \Prob{|\Bcal_n^-| = 0} \\
	&\hspace{10pt}+ \CExp{W_1(m_n,m_x^+)\ind{D_n}}{|\Bcal_n^-| \ge 1} \Prob{|\Bcal_n^-| \ge 1}
\end{align*}

For the first term we have,
\begin{align*}
	\CExp{W_1(m_n,m_x^+)\ind{D_n}}{|\Bcal_n^-| = 0} \Prob{|\Bcal_n^-| = 0} 
	&\le 2\delta_n^+ \Prob{|\Bcal_n^-| = 0} \\
	&= \bigO{\delta_n^-} e^{-n \bigT{(\delta_n^-)^N}}
		= \bigO{\frac{(\delta_n^+)^N - (\delta_n^-)^N}{(\delta_n^+)^{N-1}}},
\end{align*}
where we used that $\Exp{|\Bcal_n^-|} = n\vol_\Mcal(\Bcal_n^-) = n \bigT{(\delta_n^-)^N}$.

It now suffices to show that 
\begin{equation}\label{eq:proof_coupling_graph_discrete_measures_main2}
	\CExp{W_1(m_n,m_x^+)\ind{D_n}}{|\Bcal_n^-| \ge 1} 
	= \bigO{\frac{(\delta_n^+)^N - (\delta_n^-)^N}{(\delta_n^+)^{N-1}}}.
\end{equation}

We will do this by constructing a specific transport plan (coupling) between the measures $m_n$ and $m_x^+$. Define the joint probability mass function on $\Bcal_n \times \Bcal_n^+$
\[
	m(u,v) = \begin{cases}
		\frac{1}{|\Bcal_n^+|} &\mbox{if } u = v \\
		\frac{1}{|\Bcal_n| \, |\Bcal_n^+|}
			&\mbox{if } v \in \Bcal_n^+ \setminus \Bcal_n,
	\end{cases}
\]
and observe that $m(u,v)$ is a coupling between $m_x^G$ and $m_x^+$. Therefore
\begin{align*}
	W_1(m_x^G,m_x^+) &\le \sum_{u \in \Bcal_n} \sum_{v \in \Bcal_n^+} d_\Mcal(u,v) m(u,v)
		= \sum_{u \in \Bcal_n} \sum_{v \in \Bcal_n^+ \setminus \Bcal_n} \frac{d_\Mcal(u,v)}{|\Bcal_n| \, |\Bcal_n^+|}\\
	&\le 2\delta_n^+ \frac{|\Bcal_n^+| - |\Bcal_n|}{|\Bcal_n^+|} 
		\le 2 \delta_n^+ \frac{|\Bcal_n^+|-|\Bcal_n^-|}{|\Bcal_n^+|}
		= 2 \delta_n^+ \frac{|\Bcal_n^+ \setminus \Bcal_n^-|}{|\Bcal_n^+|}.
\end{align*}
Now define $X_n = |\Bcal_n^+ \setminus \Bcal_n^-|$ and $Y_n = |\Bcal_n^-|$. Then $X_n$ and $Y_n$ are independent Poisson random variables satisfying
\[
	\frac{|\Bcal_n^+ \setminus \Bcal_n^-|}{|\Bcal_n^+|} = \frac{X_n}{X_n + Y_n}.
\]
It then follows from Lemma~\ref{lem:technicall_poisson} that
\[
	\CExp{W_1(m_x,m_x^+)}{|\Bcal_n^-| \ge 1} \le \bigO{\frac{\delta_n^+ \Exp{X_n}}{\Exp{X_n} + \Exp{Y_n}}}
	= \bigO{\frac{\delta_n^+ \vol_\Mcal(\Bcal_n^+ \setminus \Bcal_n)}{\vol_\Mcal(\Bcal_n^+)}}.
\]
Equation~\eqref{eq:proof_coupling_graph_discrete_measures_main2} then follows by noting that $\vol_\Mcal(\Bcal_n^+ \setminus \Bcal_n) = \bigT{(\delta_n^+)^N - (\delta_n^-)^N}$.
\end{proof}

\subsection{Continuous and discrete measures on $\Mcal$}\label{ssec:proof_coupling_discrete_continuous_measures}

\subsubsection{Collecting relevant known results}

The following is a summary of results on the Wasserstein metric between empirical and uniform measures on the $N$-dimensional cube. The case $N = 2$ was explicitly stated in~\cite{talagrand1994matching}. Although the results for $N \ge 3$ are known, they are not stated in the explicit form we need. For completeness we thus include a proof here.

\begin{proposition}\label{prop:coupling_uniform_measures_cube}
Let $X_1, X_2, \dots$ be independent uniformly distributed random variables on $[0,1]^N$, let $m_n$ denote the empirical measure
\[
	m_n(y) = \frac{1}{n} \sum_{i = 1}^n \ind{X_i = y},
\]
and $\mu$ the uniform measure on $[0,1]^N$. Then
\[
	\Exp{W_1^N(m_n,\mu)} = \begin{cases}
		\bigO{\sqrt{\frac{\log(n)}{n}}} &\mbox{if } N = 2\\
		\bigO{n^{-1/N}} &\mbox{if } N \ge 3.
	\end{cases}
\]
\end{proposition}

\begin{proof}
The result for $N = 2$ follows from \cite[Equation (1.1)]{talagrand1994matching}, see also the results in~\cite{leighton1986tight} and~\cite{shor1991minimax}. For $N \ge 3$ we let $Y_1, Y_2, \dots$ be independent uniformly distributed random variables on $[0,1]^N$ and define
\[
	M_n := \inf_{\sigma} \sum_{i = 1}^n \|X_i - Y_{\sigma(i)}\|,
\]
where the infimum is taken over all permutations $\sigma$ of $\{1, 2, \dots, n\}$. Then, it follows from~\cite[Lemma 1]{talagrand1992matching} that
\[
	M_n = \sup_{f \in L_1} \left|\sum_{i = 1}^n f(X_i) - f(Y_i)\right|,
\]
where $L_1$ now denotes the set of Lipschitz continues functions with constant $1$, with respect to the Euclidean distance $d_N$. 

Next, we recall the duality formula for the Wasserstein metric on the space $\Xcal$,
\[
	W_1(\mu_1, \mu_2) = \sup_{f \in L_1} \int_\Xcal f(x) \, d\mu_1(x) - \int_\Xcal f(y) \, d\mu_2(y).
\]

Since 
\[
	\int_{[0,1]^N} f(z) \dd \mu(z) = \Exp{f(Y_i)},
\]
we have
\begin{align*}
	W_1^N(m_n, \mu) 
	&= \sup_{f \in L_1} \left|\frac{1}{n}\sum_{i = 1}^n \left(f(X_i) - \int_{[0,1]^N} f(z) \dd \mu(z)\right)\right|\\
	&= \frac{1}{n}\sup_{f \in L_1} \left|\sum_{i = 1}^n f(X_i) - \Exp{f(Y_i)}\right|\\
	&\le \frac{1}{n}\CExp{\sup_{f \in L_1} \left|\sum_{i = 1}^n f(X_i) - f(Y_i)\right|}{X_1, \dots, X_n}
		= \frac{1}{n}\CExp{M_n}{X_1, \dots, X_n},
\end{align*}
and hence
\[
	\Exp{W_1^N(m_X, \mu)} \le \frac{\Exp{M_n}}{n}.
\]
Finally~\cite[Theorem 1]{talagrand1992matching} implies for $N \ge 3$,
\[
	\Exp{M_n} = \bigO{n^{1-1/N}},
\]
which then yields
\[
	\Exp{W_1^N(m_n, \mu)} = \bigO{n^{-1/N}}.
\]
\end{proof}

\subsubsection{Uniform and discrete measures on the unit cube}

We first extend Proposition~\ref{prop:coupling_uniform_measures_cube} to the case where the points correspond to a Poisson process. We will actually proof a slightly more general version which allows for intensities $(1+\smallO{1})n$.

\begin{lemma}\label{lem:couling_discrete_uniform_measure_N_square}
Consider the $N$-dimensional unit cube $[0,1]^N$, with $N \ge 2$, and consider a Poisson process $\Pcal$ with intensity measure $(1+f_n) n \, \dd \vol_N$ on $[0,1]^N$, for some sequence $0 \le f_n \to 0$. Let $m^N_\Pcal$ denote the empirical random measure with respect to $\Pcal$, i.e. 
\[
	m_\Pcal^N(y) = \frac{1}{|\Pcal|} \sum_{p \in \Pcal} \ind{p = y},
\]
and $\mu^N$ the uniform measure on the square. Then, as $n \to \infty$,
\[
	\Exp{W_1^N(m^N_\Pcal,\mu^N)} = \bigO{\log(n) n^{-1/N}}
\]
\end{lemma}

\begin{proof}
We shall establish the result by conditioning on the size $|\Pcal|$ which has a Poisson distribution with mean $(1+f_n)n$. Conditioned on $|\Pcal| = k$, each point is uniformly distributed and therefore it follows from Proposition~\ref{prop:coupling_uniform_measures_cube} that as $k_n \to \infty$
\begin{equation}\label{eq:coupling_error_bound_conditioned_P}
	\CExp{W_1(m^N_\Pcal,\mu^N)}{\, |\Pcal|=k_n} = \begin{cases}
		\bigO{\sqrt{\frac{\log(k_n)}{k_n}}} &\mbox{if } N = 2,\\
		\bigO{k_n^{-1/N}} &\mbox{if } N\ge 3.
	\end{cases}
	\, = \bigO{\sqrt{\log(k_n)} k_n^{-1/N}}.
\end{equation}

Recall the Chernoff concentration result (\cite[Lemma 1.2]{penrose2003random}) for a Poisson random variable $\Po(a)$ with mean $a$:
\begin{equation}\label{eq:poisson_concentration}
	\Prob{\left|\Po(a) - a\right| > x} \le 2 e^{-\frac{x^2}{2(a + x)}}.
\end{equation}

Fix a $c > 0$. Then by~\eqref{eq:poisson_concentration} with $a = (1+f_n)n$ and $x = c \sqrt{(1+f_n)n \log(n)}$,
\begin{align*}
	&\hspace{-30pt}\Prob{|\Po((1+f_n)n) - (1+f_n)n| > c \sqrt{(1+f_n)n \log(n)}} \\
	&\le 2 e^{-\frac{c^2 (1+f_n)n \log(n)}{2((1+f_n)n + c \sqrt{n \log(n)})}}
		= \bigO{e^{-\frac{c^2\log(n)}{2}}} = \bigO{n^{-\frac{c^2}{2}}}.
\end{align*}
Therefore, if we define
\[
	a_n^\pm = (1+f_n)n \pm c \sqrt{(1+f_n)n \log(n)},
\]
it follows that
\begin{align*}
	\Prob{\Po((1+f_n)n) < a_n^-} &= \Prob{(1+f_n)n - \Po((1+f_n)n) \ge c \sqrt{(1+f_n)n \log(n)}} \\
	&\le \Prob{|\Po((1+f_n)n) - (1+f_n)n| > c \sqrt{(1+f_n)n \log(n)}}
		= \bigO{n^{-\frac{c^2}{2}}},
\end{align*}
and similarly
\[
	\Prob{\Po((1+f_n)n) \ge a_n^+} = \bigO{n^{-\frac{c^2}{2}}}.
\]

We shall use this and the upper bound~\eqref{eq:coupling_error_bound_conditioned_P} for $\CExp{W_1^N(\widehat{m}^N_\Pcal,\mu^N) }{ \, |P| = k_n}$ to compute an upper bound for $\Exp{W_1^N(m^N_\Pcal,\mu^N)}$ as follows:
\begin{align*}
	\Exp{W_1^N(m^N_\Pcal,\mu^N)} 
	&= \sum_{k = 0}^{a_n^- - 1} \CExp{W_1(m^N_\Pcal,\mu^N) }{\, |P| = k} \Prob{\Po(n) = k} \\
	&\hspace{10pt}+ \sum_{k = a_n^-}^{a_n^+} \CExp{W_1(m^N_\Pcal,\mu^N) }{\, |P| = k} \Prob{\Po(n) = k}\\
	&\hspace{10pt}+ \sum_{k = a_n^+ + 1}^{\infty} \CExp{W_1(m^N_\Pcal,\mu^N) }{\, |P| = k} \Prob{\Po(n) = k}\\
	&:= I_1 + I_2 + I_3.
\end{align*}
For $I_1$ we have
\[
	I_1 \le \sum_{k = 0}^{a_n^- - 1} \Prob{\Po((1+f_n)n) = k} = \Prob{\Po((1+f_n)n) < a_n^-}
	= \bigO{n^{-\frac{c^2}{2}}},
\]
while for $I_3$ we get, using~\eqref{eq:coupling_error_bound_conditioned_P},
\begin{align*}
	I_3 &\le \Prob{\Po((1+f_n)n) > a_n^+} \bigO{\sqrt{\log(a_n^+)} (a_n^+)^{-1/N}}\\
	&= \bigO{\sqrt{\log(a_n^+)} (a_n^+)^{-1/N} n^{-\frac{c^2}{2}}} 
		= \bigO{\sqrt{\log(n)} n^{-\frac{c^2}{2} - \frac{1}{N}}}.
\end{align*}
The main contribution comes from $I_2$ for which we use that $k \mapsto \Prob{\Po(Qn) = k}$ is concave on $[a_n^-, a_n^+]$ and attains is maximum at $k = (1 + f_n)n$ to obtain
\begin{align*}
	I_2 &\le \bigO{\sqrt{\log(n)} n^{-1/N}} \, \Prob{\Po((1+f_n)n) = (1+f_n)n} (a_n^+ - a_n^-)\\
	&\le \bigO{\sqrt{\log(n)} n^{-1/N}} \frac{2 (1+f_n)c \sqrt{(1+f_n)n \log(n)}}{\sqrt{2\pi} \sqrt{n}} \\
	&= \bigO{\log(n) n^{-1/N}},
\end{align*}
where we used~\eqref{eq:coupling_error_bound_conditioned_P} with $k_n = (1+f_n)n$ for the first line and Stirling's approximation for $n!$ for the second line.

Since $c > 0$ was arbitrary we conclude that
\begin{equation}\label{eq:bound_wasserstein_metric_m_ast_mu}
	\Exp{W_1^N(m^N_\Pcal,\mu^N)} = \bigO{\log(n) n^{-1/N}}.
\end{equation}
\end{proof}

\subsubsection{Uniform and discrete measures on the ball $\BallM{x}{\delta_n}$}

The following result follows from Lemma~\ref{lem:couling_discrete_uniform_measure_N_square} by a simple rescaling argument.

\begin{corollary}\label{cor:couling_discrete_uniform_measure_N_square_delta}
Let $r_n \to 0$ and consider a Poisson process $\Pcal$ with intensity $n$ on the $N$-dimensional square $[0,2 r_n]^N$. Let $m^N_\Pcal$ denote the empirical measure on the square $[0,2 r_n]^N$ with respect to $\Pcal$, i.e. 
\[
	m^N_\Pcal(y) = \frac{1}{|\Pcal \cap [0,2 r_n]^N |} \sum_{p \in \Pcal} \ind{p = y} \ind{y \in [0,2\delta_n]^N},
\]
and $\mu^N$ the uniform measure on the square $[0,2 r_n]^N$. Then
\[
	\Exp{W_1^N(m^N_\Pcal,\mu^N)} = \bigO{\log(n) n^{-1/N}}.
\]
\end{corollary}

\begin{proof}
Consider the map $\phi : [0,2 r_n]^N \to [0,1]^N$ defined by $\phi(x) = r_n^{-1} x/2$. Then $\phi(\Pcal)$ is a Poisson Point Process on $[0,1]^N$ with intensity measure $2^N r_n^N n$. Now let $\hat{m}^N_\Pcal = m^N_\Pcal \circ \phi^{-1}$ and $\hat{\mu}^N = \mu^N \circ \phi^{-1}$ denote, respectively, the empirical measure with respect to $\phi(\Pcal)$ and the uniform measure on $[0,1]^N$. It follows from Lemma~\ref{lem:couling_discrete_uniform_measure_N_square} that
\[
	\Exp{W_1^N(\hat{m}^N_\Pcal, \hat{\mu}^N)} = \bigO{\log(n r_n^N) r_n n^{-\frac{1}{N}}} 
	= \bigO{(\log(n) + N \log(r_n)) n^{-\frac{1}{N}} r_n}.
\]
Since for any $x,y \in [0,2 r_n]^N$ we have $d_N(\phi(x),\phi(y)) = 2^{-1} r_n^{-1}d_N(x,y)$ it follows that
\begin{align*}
	\Exp{W_1(m^N_\Pcal, \mu^N} &= 2^{-1} r_n^{-1} \Exp{W_1(\hat{m}^N_\Pcal, \hat{\mu}^N)} \\
	&= \bigO{\left(\log(n) + N \log(r_n)\right) n^{-\frac{1}{N}}}
		= \bigO{\log(n) n^{-1/N}},
\end{align*}
since $r_n \to 0$.
\end{proof}

For our analysis we first extend Corollary~\ref{cor:couling_discrete_uniform_measure_N_square_delta} to $N$-dimensional balls. For this we note that if $m_x^N$ and $\mu_x^N$ denote, respectively, the empirical and uniform measure on the ball $\BallN{x}{\delta_n} \subseteq \R^N$, then
\[
	W_1^N(m_x^N,\mu_x^N) \le W_1^N(m^N, \mu^N),
\]
where $m^N$ and $\mu^N$ are, respectively, the empirical and uniform measure on a cube $[0,2\delta_n]^{N}$. 
It then follows from Corollary~\ref{cor:couling_discrete_uniform_measure_N_square_delta}
\[
	\Exp{W_1(m_x^N,\mu_x^N)} = \bigO{\log(n) n^{-\frac{1}{N}}} = \smallO{\lambda_n} = \smallO{\delta_n^3}.
\]
We thus have the following result:

\begin{proposition}\label{prop:coupling_uniform_discrete_measures_balls}
Let $0 \le f_n \to 0$, $x \in \R^N$ and consider a Poisson process $\Pcal$ with intensity measure $(1+f_n) n \, \dd \vol_N$ on the $N$-dimensional ball $\BallN{x}{\delta_n}$. Let $m_x^N$ denote the empirical measure with respect to $\Pcal$, i.e. 
\[
	m_x^N(y) = \frac{1}{|\Pcal|} \sum_{p \in \Pcal} \ind{p = y},
\]
and $\mu_x^N$ the uniform measure on $\BallN{x}{\delta_n}$. Then
\[
	\Exp{W_1^N(m_x^N,\mu_x^N)} = \smallO{\delta_n^3}.
\]
\end{proposition}

\subsubsection{From the manifold to the tangent space and back}

To prove Proposition~\ref{prop:coupling_uniform_discrete_rw} we have to extend Proposition~\ref{prop:coupling_uniform_discrete_measures_balls} to the setting of Riemannian manifolds. For this we use that for $n$ large enough, the ball $\BallM{x}{\delta_n}$ can be mapped diffeomorphically by the exponential map to a slightly larger ball in the tangent space of $x$. Since the tangent space is diffemorphic to $\R^N$ we can use Proposition~\ref{prop:coupling_uniform_discrete_measures_balls} to obtain the result. However, we have to be careful since the exponential map does not preserve the metric. 

\begin{proof}[Proof of Proposition~\ref{prop:coupling_uniform_discrete_rw}]
We shall denote by $\BallN{x}{\delta}$ the ball of radius $\delta$ around $x \in \mathbb{R}^N$, according to the Euclidean metric. Fix a $0 < \xi < 1$ and pick a small enough, but fixed, neighborhood $U$ of the origin in $T_x \Mcal$ such that 1) the exponential map restricted to $U$ is a diffeomorphism, 2) there exists a constant $C > 1$ such that $U \subseteq \BallN{0}{C\delta_n}$ and 3) for any two points $y,z \in \exp(U)$
\[
	(1-\xi) d_N(\exp_x^{-1} y, \exp_x^{-1} z) \le d_\Mcal(y,z) \le (1+\xi) d_N(\exp_x^{-1} y, \exp_x^{-1} z).
\]
In particular, this implies that for $n$ large enough,
\[
	\BallN{0}{\frac{\delta_n}{1 + \xi}}
	\subseteq \exp^{-1}\{\BallM{x}{\delta_n}\} \subseteq \BallN{0}{\frac{\delta_n}{1 - \xi}}\subset U.
\]

Next we note that the probability measures $m_x^\Mcal$ and $\mu_x^{\delta_n}$ on $\BallM{x}{\delta_n}$ only depend on the restriction of the Poisson process to this ball. In particular it only depends on the restriction $\Pcal_U$ of the process to the fixed neighborhood $U$, which is again a Poisson process with intensity $\frac{n \, \dd \vol_\Mcal}{\vol_\Mcal(\Mcal)}$. Since $U \subseteq \BallN{0}{C\delta_n}$ it follows that on $U$, $\vol_\Mcal \, \circ \, \exp_x = (1+\bigO{\delta_n^2}) \vol_N$. Therefore, it follows from the Mapping Theorem for Poisson processes~\cite{last2017lectures} that $\exp_x^{-1}(P_{U})$ is a Poisson process on $\exp_x^{-1}(U)$ with intensity function $(1+\bigO{\delta_n^2}) \, \frac{n \, \dd \vol_N}{\vol_\Mcal(\Mcal)}$.

Slightly abusing notation, let $m_x^{N}$ and $\mu_x^N$ denote respectively the empirical and uniform measure on $\BallN{0}{\frac{\delta_n}{1 - \xi}}$ with respect to the Poisson Point Process $\exp_x^{-1}(\Pcal_U)$. Then, since $\delta_n/(1-\xi) = \bigT{\delta_n}$, Proposition~\ref{prop:coupling_uniform_discrete_measures_balls} implies that
\[
	\Exp{W_1^N(m_x^N,\mu_x^N)} = \smallO{\delta_n^3}.
\]

On the other hand we have, since $\exp_x$ is a diffeomorphism on $U$, that
\[
	\Exp{W_1(m_x^\Mcal,\mu_x^{\delta_n})} \le (1+\xi) \Exp{W_1^N(m_x^N,\mu_x^N)},
\]
and hence we conclude that
\[
	\Exp{W_1(m_x^\Mcal,\mu_x^{\delta_n})} = \smallO{\delta_n^3},
\]
which proves Proposition~\ref{prop:coupling_uniform_discrete_rw}.
\end{proof}

\subsection{Weighted graph distances}\label{sec:weighted_graph_distances}

Recall that $\lambda_n = \log(n)^{\frac{2}{N}}n^{-\frac{1}{N}}$. To prove Proposition~\ref{prop:weighted_graph_distance_delta_approx} we first show the following

\begin{lemma}\label{lem:adjusted_manifold_distance}
Let $Q > 3$, $U = \BallM{x^\ast}{Q\delta_n}$ and define the event
\[
	A_n := \bigcup_{u,v \in U \cap G_n} \left\{\left|d_G^w(u,v) - d_\Mcal(u,v) \right| > d_\Mcal(u,v) \frac{3 \lambda_n}{\varepsilon_n} + 2\lambda_n\right\}.
\]
Then, $\Prob{A_n} = \smallO{\delta_n^3}$, as $n \to \infty$.
\end{lemma}

\begin{proof}

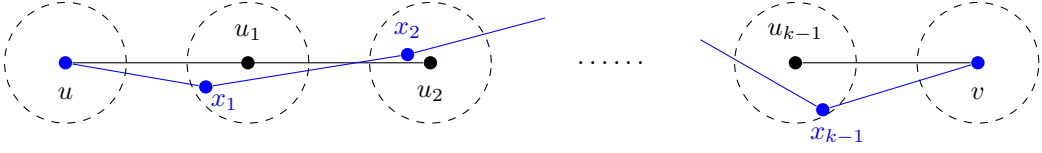
\begin{figure}
\centering
\begin{tikzpicture}[scale=0.8]
	\tikzstyle{vertex}=[fill, circle, inner sep=0pt, minimum size=5pt]
	\tikzstyle{edge}=[color=black]
	
	\draw node[vertex,blue] (u) at (0,0) {};
	\draw node[vertex] (z1) at (3,0) {};
	\draw node[vertex] (z2) at (6,0) {};
	\draw node[vertex] (z4) at (12,0) {};
	\draw node[vertex,blue] (v) at (15,0) {};
	
	\path (z1)+(210:0.8) node[vertex,blue] (x1) {}; 
	\path (z2)+(160:0.4) node[vertex,blue] (x2) {}; 
	\path (x2)+(15:2.5) node (x31) {};
	
	\path (z4)+(300:0.9) node[vertex,blue] (x4) {};
	\path (x4)+(150:2.5) node (x32) {};
	
	\path (u)+(270:0.5) node {$u$};
	\path (z1)+(90:0.5) node {$u_1$};
	\path (z2)+(270:0.5) node {$u_2$};
	\path (z4)+(90:0.5) node {$u_{k-1}$};
	\path (v)+(270:0.5) node {$v$};

	\path (x1)+(320:0.4) node {\color{blue} $x_1$};
	\path (x2)+(90:0.4) node {\color{blue} $x_2$};
	\path (x4)+(300:0.5) node {\color{blue} $x_{k-1}$};
	
	\draw (u) -- (z2);
	\draw node at (9,0) {$\dots\dots$};
	\draw (z4) -- (v);
	\draw[blue] (u) -- (x1) -- (x2) -- (x31);
	\draw[blue] (x32) -- (x4) -- (v);
	
	\draw[dashed] (u) circle  (1cm);
	\draw[dashed] (z1) circle  (1cm);
	\draw[dashed] (z2) circle  (1cm);
	\draw[dashed] (z4) circle  (1cm);
	\draw[dashed] (v) circle  (1cm);

\end{tikzpicture}
\caption{Depiction of the splitting of the geodesic between $u$ and $v$ in $k$ equal segments.}
\label{fig:partition_geodesic_uv}
\end{figure}
The proof closely follows the strategy of the proof of Lemma~\ref{lem:extended_distance_metric_space}. Let $C_n$ denote the event in Corollary~\ref{cor:covering}. We will show that on this event,
\[
	\left|d_G^w(u,v) - d_\Mcal(u,v) \right| \le \frac{3 d_\Mcal(u,v) \lambda_n}{\varepsilon_n} + 2\lambda_n,
\]
for all $u,v \in U \cap G_n$. This then implies that $\Prob{A_n, C_n} = 0$ from which the results follows, since by Corollary~\ref{cor:covering}
\[
	\Prob{A_n} \le \Prob{A_n, C_n} + (1-\Prob{C_n}) = \smallO{\delta_n^3}.  
\]

Take any two $u,v \in U \cap G_n$ and let $\gamma(u,v)$ denote the geodesic between $u$ and $v$. We then partition this geodesic into 
\[
	k = \left\lceil \frac{3d_\Mcal(u,v)}{\varepsilon_n} \right\rceil \le \frac{3d_\Mcal(u,v)}{\varepsilon_n} + 1.
\]
pieces of equal length and let $u := u_0, u_1, \dots, u_{k-1}, u_k := v$ denote the $k + 1$ endpoints of the intervals, see Figure~\ref{fig:partition_geodesic_uv}. On the event $C_n$, each $u_t$ belongs to some ball $B_t$ of radius $ \lambda_n/4$ which contains a vertex $x_t \in G$, where we can take $x_0 = u$ and $x_k = v$. In particular, since $d_\Mcal(u_t,x_t) \le \lambda_n/2$, $d_\Mcal(u_{t-1},u_{t}) \le \varepsilon_n/3$ and $\lambda_n = \smallO{\varepsilon_n}$, it follows that for large enough $n$, 
\[
	d_\Mcal(x_t,x_{t+1}) \le d_\Mcal(u_t,x_t) + d_\Mcal(u_{t+1},x_{t+1}) + d_\Mcal(u_t,u_{t+1}) 
	\le \lambda_n + \frac{\varepsilon_n}{3} \le \varepsilon_n
\] 
so that $\{u, x_1, \dots, x_k,v\}$ is a path in $G_n$ (see Figure~\ref{fig:partition_geodesic_uv}). Moreover, $d_{G}^w(x_t,x_{t+1}) \le d_\Mcal(u_t,u_{t+1}) + \lambda_n$ by the triangle inequality. Therefore,
\begin{align*}
	d_G^w(u,v) &\le \sum_{t = 0}^{k-1} d_{G}^w(x_t, x_{t+1})
		\le \sum_{t = 0}^{k-1} \left(d_\Mcal(u_t, u_{t+1}) + \lambda_n \right)\\
	&\le d_\Mcal(u,v) + k \lambda_n
		\le d_\Mcal(u,v)\left(1 + \frac{3 \lambda_n}{\varepsilon_n}\right) + \lambda_n.
\end{align*}
To finish the proof we note that by definition $d_G^w(u,v) \ge d_\Mcal(u,v)$ and hence
\[
	\left|d_G^w(u,v) - d_\Mcal(u,v)\right|
	= d_G^w(u,v) - d_\Mcal(u,v)
	\le \frac{3 d_\Mcal(u,v) \lambda_n}{\varepsilon_n} + 2\lambda_n.
\]
\end{proof}

\begin{proof}[Proof of Proposition~\ref{prop:weighted_graph_distance_delta_approx}]
Due to Lemma~\ref{lem:adjusted_manifold_distance} it suffices to show that the conditions on $\varepsilon_n$ and $\delta_n$ imply $\frac{\lambda_n}{\varepsilon_n} = \smallO{\delta_n^2}$.

We compute that
\[
	\frac{\lambda_n}{\varepsilon_n \delta_n^2} = \bigT{\log(n)^{\frac{2}{N} - a - 2b} n^{\alpha + 2\beta-\frac{1}{N}}}.
\]
The latter is $\smallO{1}$ precisely when either $\alpha + 2\beta < \frac{1}{N}$ or $\alpha + 2\beta = \frac{1}{N}$ and $a + 2b > \frac{2}{N}$, which are the conditions of Proposition~\ref{prop:weighted_graph_distance_delta_approx}.

Thus, under the conditions of Proposition~\ref{prop:weighted_graph_distance_delta_approx}
it holds that the manifold-weighted graph distance $d_G^w$ is a $\delta_n$-good approximation with $\xi_n = \max\{\sqrt{\lambda_n/\varepsilon_n}, \lambda_n^{1/3}\}$.
\end{proof}

\subsection{Rescaled graph distances}\label{sec:almost_sparse_graphs}

Consider the $2$-dimensional Euclidean space equipped with the Euclidean distance $d_2$. Let $\Ccal = [0,1]^2$ and take $G_n = \mathbb{G}_n(\varepsilon)$ to be the random geometric graph on $\Ccal$ with connection radius $\varepsilon$. The main result in~\cite{diaz2016relation} relates the shortest-path distance $d_{G_n}^s$ and the Euclidean distance $d_2$. We state a version of this result here, which includes the error bounds that follow from Proposition~2.2 and Proposition~2.4 in~\cite{diaz2016relation}.

\begin{theorem}[Theorem 1.1 from~\cite{diaz2016relation}]\label{thm:graph_distance_vs_euclidean}
Consider the random geometric graph $G_n$ on the unit square $[0,1]^2$ with connection radius $\varepsilon_n = \smallO{1}$. Then for any pair of vertices $x,y \in G_n$ with $d_2(x,y) > \varepsilon_n$, the following holds:
\begin{enumerate}
\item If $d_2(x,y) \ge \max\left\{\frac{12 \log(n)^{3/2}}{n \varepsilon_n}, 21 \varepsilon_n \log(n)\right\}$, then
\[
	\Prob{d_G^s(x,y) \ge \left\lfloor\frac{d_2(x,y)}{\varepsilon_n}\left(1 + \frac{1}{2(n \varepsilon_n d_2(x,y))^{2/3}}\right) \right\rfloor} \ge 1 - \smallO{n^{-5/2}}. 
\]
\item If $\varepsilon_n \ge 224 \sqrt{\log(n)/n}$ then
\[
	\Prob{d_G^s(x,y) \le \left\lceil \frac{d_2(x,y)}{\varepsilon_n}\left(1 + \gamma_n\right) \right\rceil}
	\ge 1 - \smallO{n^{-5/2}}
\]
with
\[
	\gamma_n := \max\left\{1358 \left(\frac{3 \log(n)}{n\varepsilon_n^2 + n \varepsilon_n d_2(x,y)}\right)^{2/3}
	\hspace{-10pt}, \quad
	\frac{4 \times 10^6 \log(n)^2}{n^2 \varepsilon_n^4}, \quad
	\left(\frac{30000}{n \varepsilon_n^2}\right)^{2/3}
	\right\}.
\]
\end{enumerate}
\end{theorem}

From this we obtain the following result, which gives bounds on the graph distance $\varepsilon_n d_G^s$ in terms of the manifold distance, between two nodes of the graph $G_n$ that are within manifold distance $\bigO{\delta_n}$.

\begin{lemma}\label{lem:shortest_path_distance}
Let $\varepsilon_n \ge 244 \sqrt{\log(n)/n}$, $Q > 3$, $U = \BallM{x^\ast}{Q\delta_n}$ and define the event
\[
	A_n := \bigcup_{u,v \in U \cap G_n} \left\{\left|\varepsilon_n d_G^s(u,v) - d_\Mcal(u,v) \right| > d_\Mcal(u,v)\gamma_n + \varepsilon_n \right\}.
\]
Then, $\Prob{A_n} = \smallO{\delta_n^3}$, as $n \to \infty$.
\end{lemma}

\begin{proof}
Note that since the the neighborhood $U$ is shrinking as $n$ increases we can map it to $\mathbb{R}^2$ diffeomorphically for sufficiently large $n$. This affects the distances at most by a constant factor and hence it suffices to prove the statement for $\Mcal = \mathbb{R}^2$.

By the second statement of Theorem~\ref{thm:graph_distance_vs_euclidean} we have that for any two $u,v \in U \cap G_n$
\[
	\Prob{\left|\varepsilon_n d_G^s(u,v) - d_\Mcal(u,v) \right| > d_\Mcal(u,v)\gamma_n + \varepsilon_n} = \smallO{n^{-\frac{5}{2}}}.
\]
By conditioning on the number of nodes in $U$ ($|U \cap G_n|$) and applying the union bound we get
\[
	\CProb{A_n}{|U \cap G_n|} \le |U \cap G_n| \smallO{n^{-\frac{5}{2}}}.
\]
Now $\Exp{|U \cap G_n|} = \bigT{n \delta_n^2}$ and therefore 
\begin{align*}
	\Prob{A_n} &= \Exp{\CProb{A_n}{|U \cap G_n|}} \le \bigO{n^{-\frac{3}{2}} \delta_n^2} = \smallO{\delta_n^3},
\end{align*}
where we used that $n^{-\frac{3}{2}} = \smallO{\delta_n}$ for all $\delta_n = \bigT{\log(n)^b n^{-\beta}}$ and $\beta \le 1$.
\end{proof}

We can now proof Proposition~\ref{prop:shortest_graph_distance_delta_approx}.

\begin{proof}[Proof of Proposition~\ref{prop:shortest_graph_distance_delta_approx}]
First observe that $\varepsilon_n d_G^s(u,v) \ge d_\Mcal(u,v)$ for all $u,v \in \BallM{x^\ast}{Q \delta_n}$. Moreover, the conditions of the proposition imply that $\log(n)^{1/2} n^{-\frac{1}{2}} = \smallO{\varepsilon_n}$. Therefore, by Lemma~\ref{lem:shortest_path_distance} we have that with probability $1 - \smallO{\delta_n^3}$,
\[
	\left|\varepsilon_n d_G^s(u,v) - d_\Mcal(u,v) \right| \le d_\Mcal(u,v)\gamma_n + \varepsilon_n
\]
for all $u,v \in \BallM{x^\ast}{Q\delta_n} \cap G_n$. Moreover, since by assumption $\alpha \ge 3\beta$ and $a < 3b$ if $\alpha = 3\beta$ it follows that $\varepsilon_n = \smallO{\delta_n^3}$. Thus, to prove Proposition~\ref{prop:shortest_graph_distance_delta_approx} it remains to show that $\gamma_n = \smallO{\delta_n^2}$. 

Since $\gamma_n$ is the maximum of three terms
\begin{align*}
	1358 \left(\frac{3 \log(n)}{n\varepsilon_n^2 + n \varepsilon_n d_2(x,y)}\right)^{2/3} \hspace{-10pt}, \quad
	\frac{4 \times 10^6 \log(n)^2}{n^2 \varepsilon_n^4} \quad \text{and} \quad
	\left(\frac{30000}{n \varepsilon_n^2}\right)^{2/3} \hspace{-10pt}.
\end{align*}
We will show that each of them is $\smallO{\delta_n^2}$.

For the first term it suffices to show that $\log(n) n^{-1} \varepsilon_n^{-2} = \smallO{\delta_n^3}$. This follows since
\[
	\log(n) n^{-1} \varepsilon_n^{-2} \delta_n^{-3} = \bigO{\log(n)^{1 -2a - 3b} n^{-(1 - 2\alpha - 3\beta)}}
\]
which is $\smallO{1}$ by the assumption that $2\alpha + 3\beta \le 1$ and $2a + 3b > 1$ if $2\alpha + 3\beta = 1$. We now immediately have that $\left(\log(n) n^{-1} \varepsilon_n^{-2}\right)^2 = \smallO{\delta_n^6}$, which proves that the second term is $\smallO{\delta_n^2}$. Finally, the result for the third term follows from $n^{-1} \varepsilon_n^{-2} = \smallO{\log(n) n^{-1} \varepsilon_n^{-2}} = \smallO{\delta_n^3}$.
\end{proof}

\subsection*{Acknowledgments}

We thank J.~Jost and R.~Loll for useful discussions, suggestions, and comments.
This work was supported by ARO Grant Nos.~W911NF-16-1-0391 and W911NF-17-1-0491, and by NSF Grant Nos.~IIS-1741355 and DMS-1800738. 

\bibliographystyle{abbrvnat}
\bibliography{references}

\end{document}